%% file: sl2c.tex
\documentclass[11pt,a4paper,english]{article}

\usepackage{amsmath}
\usepackage{amsthm}
\usepackage{amssymb}
\usepackage{amsopn}
\usepackage{multicol}     
\usepackage{tikz-cd}
\usepackage{pgfplots}
\usepackage{wrapfig}
\usepackage{mathrsfs}
\usepackage[all]{xy}
\usepackage{mathtools,xparse}
\usepackage{theoremref}
\usepackage{combelow}
\usepackage{pifont}
\usepackage{pdfpages}
\usepackage{dsfont}
\usepackage{textgreek}
\usepackage{stackrel,amssymb}
\usepackage{enumerate}

\usepackage{hyperref}

\usepackage{tabularx}
\usepackage{array}
\usepackage{mathtools}
\usepackage{etoolbox}
\usepackage{url}

\usepackage[normalem]{ulem}

\DeclareFontFamily{U}{matha}{\hyphenchar\font45}
\DeclareFontShape{U}{matha}{m}{n}{
<-6> matha5 <6-7> matha6 <7-8> matha7
<8-9> matha8 <9-10> matha9
<10-12> matha10 <12-> matha12
}{}
\DeclareSymbolFont{matha}{U}{matha}{m}{n}

\DeclareFontFamily{U}{mathx}{\hyphenchar\font45}
\DeclareFontShape{U}{mathx}{m}{n}{
<-6> mathx5 <6-7> mathx6 <7-8> mathx7
<8-9> mathx8 <9-10> mathx9
<10-12> mathx10 <12-> mathx12
}{}
\DeclareSymbolFont{mathx}{U}{mathx}{m}{n}

\DeclareMathDelimiter{\vvvert} {0}{matha}{"7E}{mathx}{"17}%

\DeclareMathOperator{\N}{\mathbb{N}}
\DeclareMathOperator{\Z}{\mathbb{Z}}
\DeclareMathOperator{\R}{\mathbb{R}} 
\DeclareMathOperator{\C}{\mathbb{C}}

\DeclareMathOperator{\spl}{\mathfrak{sl}_2(\C)}

\DeclareMathOperator{\so}{\mathfrak{so}(3,1)}

\DeclareMathOperator{\g}{\mathfrak{g}}

\DeclareMathOperator{\id}{Id}

\DeclareMathOperator{\dif}{\mathrm{d}}

\DeclareMathOperator{\Lie}{\mathcal{L}}

\newtheorem{theorem}{Theorem}

\newtheorem{proposition}{Proposition}[section]

\newtheorem{lemma}[proposition]{Lemma}
\newtheorem{corollary}[proposition]{Corollary}
\newtheorem{remark}[proposition]{Remark}

\DeclareDocumentCommand{\norm}{ s m }{%
	\IfBooleanTF{#1}
	{#2}
	{\lVert{#2}\rVert}%
}
\DeclareDocumentCommand{\fnorm}{ s m }{%
	\IfBooleanTF{#1}
	{#2}
	{[]{#2}[]}%
}

\newcommand{\Addresses}{{
  \bigskip
  \footnotesize

Ioan M\u{a}rcu\cb{t},\par\nopagebreak
\textsc{Radboud University Nijmegen, 6500 GL Nijmegen, The Netherlands}\par\nopagebreak
  \textit{E-mail address}: \texttt{i.marcut@math.ru.nl}

\medskip

 Florian Zeiser,\par\nopagebreak
\textsc{University of Illinois at Urbana-Champaign, 61801 Urbana, United States}\par\nopagebreak
\textit{E-mail address}: \texttt{fzeiser@illinois.edu}
}}

\title{The Poisson linearization problem for $\mathfrak{sl}_2(\C)$\\
[1ex] \Large Part II: The Nash-Moser method}
\author{Ioan M\u{a}rcu\cb{t} and Florian Zeiser}

\pgfplotsset{compat=1.16}

\usepackage{mathtools,xparse}

\DeclarePairedDelimiter{\oldnormaux}{\bracevert}{\bracevert}

\NewDocumentCommand{\oldnorm}{som}{%
  \IfBooleanTF{#1}
    {\oldnormaux*{#3}}
    {\IfNoValueTF{#2}
       {\oldnormaux*{\vphantom{dq}#3}}
       {\oldnormaux[#2]{#3}}%
    }%
}

\begin{document}
\maketitle

\begin{abstract}This is the second of two papers, in which we prove a version of Conn's linearization theorem for the Lie algebra $\mathfrak{sl}_2(\C)\simeq \so$. Namely, we show that any Poisson structure whose linear approximation at a zero is isomorphic to the Poisson structure associated to $\mathfrak{sl}_2(\C)$ is linearizable. 

In the first part, we calculated the Poisson cohomology associated to $\mathfrak{sl}_2(\C)$, and we constructed bounded homotopy operators for the Poisson complex of multivector fields that are flat at the origin.

In this second part, we obtain the linearization result, which works for a more general class of Lie algebras. For the proof, we develop a Nash-Moser method for functions that are flat at a point.
\end{abstract}

\setcounter{theorem}{3}

\input{Intro_sl2c.tex}

\tableofcontents
\setcounter{section}{4}

\input{Linearization}

\appendix
\setcounter{section}{2}
\input{Setting_NM}

\input{Smoothing}

\bibliographystyle{alpha}
\bibliography{Bibtex}

\Addresses

\end{document}

%% file: Intro_sl2c.tex
\section*{Introduction to Part II}

The Nash-Moser method, developed in \cite{Nash} and \cite{Mosera,Moserb}, is a powerful tool for solving PDEs, which can be understood as an inverse function theorem for non-linear differential operators whose linearization is invertible, but the inverse ``loses derivatives''. 

In order to apply the Nash-Moser method to the problem of linearization for Poisson structures with linear part corresponding to a Lie algebra $\g$, as was done by Conn in the compact semisimple case \cite{Conn85}, one would need linear operators:
\begin{align}\label{eq: homotopy maps:general:Lie algebra}
    \mathfrak{X}^1( \overline{B}_R) \xleftarrow{h^1} \mathfrak{X}^2( \overline{B}_R)\xleftarrow{h^2}\mathfrak{X}^3( \overline{B}_R),
\end{align} 
where $\overline{B}_R\subset \g^*$ is the closed ball of radius $R>0$ around the origin, which satisfying the homotopy equation:
\begin{align}\label{eq: homotopy relation}
    W = h^2\circ \dif_{\pi_{\g}}(W) +\dif_{\pi_{\g}}\circ h^1(W),
\end{align}
for all $W\in \mathfrak{X}^2(\overline{B}_R)$, and, for some fixed $d\in \N_0$, inequalities of the sort: 
\begin{align}\label{eq: tame estimate}
    \norm{h^{i}(W)}_{n,R}\le C\cdot \norm{W}_{n+d,R},
\end{align}
for all $W\in \mathfrak{X}^{i+1}(\overline{B}_R)$, all $n\in \N_0$ and some constants $C=C(n)$. Here $\norm{\cdot}_{n,R}$ are usual $C^n$-norms on multivector fields on $\overline{B}_R$. These types of estimates are standard for applying the Nash-Moser method; see e.g.\ \cite{Ham}, where they are called ``tame'' estimates.  

In Theorem C of Part I of the paper, we constructed explicit homotopy operators for the Lie algebra $\spl$. In contrast to the compact semisimple case from Conn's paper, our operators are defined only on the subcomplex of multivector fields that are flat at the origin: 
\begin{align}\label{eq: homotopy maps}
    \mathfrak{X}^1_0( \mathfrak{sl}_2(\C)^*) \xleftarrow{h^1} \mathfrak{X}^2_0( \mathfrak{sl}_2(\C)^*)\xleftarrow{h^2}\mathfrak{X}^3_0( \mathfrak{sl}_2(\C)^*)
\end{align} 
These still satisfy the homotopy relation \eqref{eq: tame estimate}, for all flat bivector fields $W\in \mathfrak{X}^{i+1}_0( \mathfrak{sl}_2(\C)^*)$.  Moreover, as shown in Part I, they satisfy the SLB property, meaning that they induced homotopy operators on the balls of radius $R>0$:
\begin{align}\label{eq: homotopy maps for r} \mathfrak{X}^1_0(\overline{B}_R) \xleftarrow{h^1} \mathfrak{X}^2_0( \overline{B}_R)\xleftarrow{h^2}\mathfrak{X}^3_0( \overline{B}_R),
\end{align}
which satisfy, for some integers $a,b,c\in \N_0$, estimates of the following type:
\begin{align}\label{eq: property H estimate}
    |h^i(W)|_{n,k,R}\le C\, |W|_{n+a,k+bn+c,R},
\end{align}
for all $W\in \mathfrak{X}^{i+1}(\overline{B}_R)$
and all $n,k\in \N_0$, where $C=C(R,n,k)$ depends continuously on $R$. The double-indexed norms $|\cdot|_{n,k,R}$ are natural to the space of multivector fields that are flat at the origin, and are defined as a type of weighted $C^n$-norms: \begin{align}\label{eq: flat norms 2}
    |W|_{n,k,R}:= \norm{\frac{1}{|x|^k}W}_{n,R},
\end{align}
for $W\in \mathfrak{X}^{i}_0(\overline{B}_r)$. We will show in Subsection \ref{subsec: equivalent seminorms}, that these norms are equivalent to the norms $\fnorm{\cdot}_{n,k,R}$, defined in Part I. 

The main result of Part II is the following linearization result for Poisson structures with linear Taylor expansion: 
\begin{theorem}\label{theorem: rigidity:qualitative}
Let $(\mathfrak{g}^*,\pi_{\mathfrak{g}})$ be the linear Poisson structure associated to the Lie algebra $\mathfrak{g}$. Assume that, for some closed ball $\overline{B}_R\subset \mathfrak{g}^*$ around 0, there exist linear operators on the flat Poisson complex over $\overline{B}_R$, as in \eqref{eq: homotopy maps for r}, satisfying the homotopy relation \eqref{eq: homotopy relation} and the estimates \eqref{eq: property H estimate}, for some integers $a,b,c\in \N_0$. Then for any Poisson structure $\pi$ defined on $\overline{B}_R$ and satisfying
\[ j^{\infty}_0(\pi-\pi_{\mathfrak{g}})=0,\]
there exists $r>0$ and an open Poisson embedding
\[ \Phi: (B_r,\pi)\hookrightarrow (B_R,\pi_{\mathfrak{g}})\]
such that $\Phi(0)=0$ and $j^{\infty}_0(\Phi-\id)=0$.
\end{theorem}

In Section \ref{section: flat linearization} we will use Theorem C from Part I, Theorem \ref{theorem: rigidity:qualitative}, and Weinstein's formal linearization theorem for semisimple Lie algebras \cite{Wein83}, to deduce Theorem A in Part I, i.e., that $\spl$ is a Poisson non-degenerate Lie algebra. Of course, the same argument implies Conn's theorem \cite{Conn85}, and can be used for any semisimple Lie algebra $\g$ for which the Poisson linearization question is still open (see the Introduction to Part I), provided the second Poisson cohomology of $\g^*$ vanishes and one can construct homotopy operators on its flat Poisson complex that satisfy estimates of the type \eqref{eq: property H estimate}. 

\medskip

We will use the Nash-Moser method to prove Theorem \ref{theorem: rigidity}, which is a quantitative and stronger version of Theorem \ref{theorem: rigidity:qualitative}. For this, we had to make specific adaptation of the method, which we explain here. First, note that the estimates \eqref{eq: property H estimate} together with the trivial estimates $|\cdot|_{n,k,r}\leq \norm{\cdot}_{n+k,r}$ (Corollary \ref{corollary: degree shift}) imply that, for some $p$ and $d$, our operators satisfy:
\begin{align}\label{eq: bad estimate}
    \norm{h^{i}(W)}_{n,r} \le C\cdot \norm{W}_{p\cdot n +d,r}.
\end{align} 
It has been shown in \cite{LoZe} that the Nash-Moser method still works under these assumptions for $p<2$, and that the method fails for $p\geq 2$, which unfortunately is the case for us. In order to make better use of inequalities \eqref{eq: property H estimate}, note that the loss of derivatives is ``tame'' in $n$, but the price we have to pay is accumulated in the weight $k$. We manage to use the Nash-Moser method as in \cite{LoZe} for the estimates of the form \eqref{eq: property H estimate} by using the following idea. For a large $q$, if we consider the family of norms 
\[ \oldnorm{W}_{n,r}:=|W|_{n,q\cdot n,r},\]
then the estimates \eqref{eq: property H estimate} imply estimates of type \eqref{eq: bad estimate}:
\[\oldnorm{ h^{i}(W)}_{n,r} \le C\cdot \oldnorm{ W}_{p\cdot n +d,r}, \]
where $p\in (1,2)$, as in \cite{LoZe} (e.g.\ $p=\tfrac{3}{2}$, for $q=2b$, $d=a+c$). 

The other crucial ingredient in the Nash-Moser method is the use of a family of smoothing operators $\{S_{t}\}_{t>1}$ which approximate the identity operator when $t\to \infty$. In the classical setting, such operators are obtained by convolution with mollifiers, and their construction goes back to the work of Nash \cite{Nash}. However, the smoothing operators we found in the literature do not preserve the space of flat functions. We managed to construct smoothing operators which do preserve flatness by making use of the following observation: under the inversion map, $x\mapsto x/|x|^2$, flatness at $0$ corresponds to Schwartz-like behavior at infinity. In Appendix \ref{section: functional analysis} we build a 2-parametric family of smoothing operators corresponding to the 2-parametric family of weighted norms \eqref{eq: flat norms 2} on flat multivector fields.

\medskip

The Nash-Moser method for flat functions, developed in the paper, should be useful in many other settings. For example, several local form problems in differential geometry can be split into a formal part and a flat part. The formal problem is at the level of formal power series, and can often be treated using algebraic methods. The flat problem is of an analytic nature, and can potentially be approached with the tools developed here.

%% file: Linearization.tex
\section{A quantitative linearization theorem}\label{section: flat linearization}

First, we show how to obtain the main result of this paper, namely that $\spl$ is Poisson non-degenerate, from the results of Part I and Theorem \ref{theorem: rigidity:qualitative}:

\begin{proof}[Proof: (Theorem C $+$ Theorem \ref{theorem: rigidity:qualitative}) $\implies$ Theorem A]
Since $\spl$ is semi\-simple, $\pi_{\spl}$ is formally non-degenerate \cite[Theorem 6.1]{Wein83}. Theorem C provides the homotopy operators needed to apply Theorem \ref{theorem: rigidity:qualitative}, and obtain Theorem A.
\end{proof}

Note that Theorem \ref{theorem: rigidity:qualitative} gives no information about the size of the ball on which the Poisson structure $\pi$ is linearizable. Below we give a more quantitative version of the result, which states that the size of this ball can be predetermined, if $\pi$ is close enough to $\pi_{\mathfrak{g}}$. It will be this version of the theorem that we will prove with the Nash-Moser technique. 

\begin{theorem}\label{theorem: rigidity}
Under the assumptions of Theorem \ref{theorem: rigidity:qualitative} the following holds. There exists $N\geq 0$ such that for any $r\in (0,R)$ there exists $\delta(r)>0$ such that for every Poisson structure $\pi$ on $\overline{B}_R$ satisfying
\[ j^{\infty}_0(\pi-\pi_{\mathfrak{g}})=0 \ \ \ \ \ \text{ and } \ \ \ \ \ \norm{\pi-\pi_{\mathfrak{g}}}_{N,R}\le \delta(r),\]
there exists an open Poisson embedding 
\[ \Phi: (B_r,\pi)\hookrightarrow (B_R,\pi_{\mathfrak{g}})\]
satisfying $\Phi(0)=0$ and $j^{\infty}_0(\Phi-\id)=0$.
\end{theorem}

We show now that it will be sufficient to prove this result:

\begin{proof}[Proof: Theorem \ref{theorem: rigidity} $\implies$ Theorem \ref{theorem: rigidity:qualitative}]
Let $m_t:\overline{B}_{R}\to \overline{B}_{t R}$ denote scalar multiplication by $t>0$, i.e.\ $m_t(x):=tx$. For $t\in (0,1]$, define a path of Poisson bivectors $\pi_t\in \mathfrak{X}^2(\overline{B}_R)$ via
\[ \pi_t:= t\cdot m_{t}^*\pi.\]
 Note that the path extends smoothly at $t=0$ with $\pi_0=\pi_{\mathfrak{g}}$. By continuity, there exists a $t>0$ such that 
\[ \norm{\pi_{t}-\pi_{\mathfrak{g}}}_{N,R}\le \delta(R/2) .\]
Therefore, by Theorem \ref{theorem: rigidity} for $r=R/2$, there exists an open embedding $\Phi:B_{R/2}\hookrightarrow B_{R}$ such that $\Phi^*(\pi_{t})=\pi_{\g}$. Using linearity of $\pi_{\g}$ we have:
\[ \pi_{\g}=\frac{1}{t} m^*_{\frac{1}{t}}\pi_{\g}=\frac{1}{t}m_{\frac{1}{t}}^*(\Phi^*(t\cdot m_t^*(\pi)))=(m_t\circ \Phi\circ m_{\frac{1}{t}})^*(\pi).\]
Hence we obtain an open Poisson embedding 
\[ m_{t} \circ \Phi \circ m_{\frac{1}{t}}:(B_{tR/2},\pi_{\g})\hookrightarrow (B_{tR},\pi)\]
which fixes the origin to infinite order. 
\end{proof}

The homotopy operators described in the following remark will be used in the proof of Theorem \ref{theorem: rigidity}.

\begin{remark}\label{remark: homotopy ops} Theorems \ref{theorem: rigidity:qualitative} and \ref{theorem: rigidity} assume the existence of homotopy operators only on the ball $\overline{B}_R$. However, because the Poisson structure is linear, by using rescaling, we obtain homotopy operators on each ball of radius $r>0$:
\begin{align*}
    \mathfrak{X}^1_0(\overline{B}_r)\xleftarrow{h^1_r} \mathfrak{X}^2_0(\overline{B}_r)\xleftarrow{h^2_r} \mathfrak{X}^3_0(\overline{B}_r)
\end{align*}
where
\begin{align}\label{eq: homotopy on general ball}
    h^i_r:= \frac{r}{R}\cdot m^*_{\frac{R}{r}}\circ h^i\circ m^*_{\frac{r}{R}}.
\end{align}
These operators also satisfy inequalities as in the statement, with constants $C(n,k,r)$ depending continuously on $r>0$. Note however that these operators might not satisfy the SLB, because they are not necessarily compatible with the restriction maps, i.e., $h^{i}_r(W|_r)$ might be different from $h^i_R(W)|_r$, for $W\in \mathfrak{X}^{i+1}_0(\overline{B}_R)$ and $R>r$. 
\end{remark}

\section{The sketch of the Nash-Moser algorithm}\label{section: flat NM}
We explain schematically the algorithm. Given $\pi$, our aim is to find a diffeomorphism $\Phi$ such that $\Phi^*(\pi)=\pi_{\g}$. Consider the map
\begin{equation*}
    \begin{array}{ccc}
        (\pi_{\mathfrak{g}}+\mathfrak{X}^2_0(\g^*))\times \mathfrak{X}^1_0(\g^*)& \to & \mathfrak{X}^2_0(\g^*)  \\
         (\pi=\pi_{\mathfrak{g}}+Z,X)&\mapsto &\phi_X^*(\pi)-\pi_{\g} 
    \end{array}
\end{equation*}
where $\phi_X$ denotes the time-one flow of $X$. We think about the problem as follows: given $Z$ we need to find $X$ which minimizes the output of this map. For the sake of this argument, we assume that we work in a Banach space with norm $|\cdot |$. Moreover, we assume that $Z$ is small in this norm, and we are looking for an $X$ which is also small. We denote this schematically as:
\[ |Z|\leq \mathcal{O}(\epsilon),\qquad |X|\leq \mathcal{O}(\epsilon).\]
Then the pullback along the flow $\phi_X$ can be expanded as:
\[ \phi_X^*=\id +\Lie_X +\mathcal{O}(\epsilon^2).\]
This allows us to estimate 
\begin{align*}
    |\phi_X^*(\pi)-\pi_{\mathfrak{g}}|\le &\ |\phi_X^*(\pi_{\mathfrak{g}})-\pi_{\mathfrak{g}}-\Lie_X \pi_{\mathfrak{g}}|+|\phi_X^*(Z)-Z|+|Z-\dif_{\pi_{\mathfrak{g}}}X|\\
    \le &\ |Z-\dif_{\pi_{\mathfrak{g}}}X| +\mathcal{O}(\epsilon^2).
\end{align*}

The cohomological flavour of the problem is now evident: we would like to find a ``primitive'' of $Z$. Assume that $H^2(\g^*,\pi_{\g})=0$, and that we have continuous homotopy operators $h^i:\mathfrak{X}_0^{i+1}(\g^*)\to \mathfrak{X}_0^i(\g^*)$, $i=1,2$. 

Then $X:=h^1(Z)$ satisfies $X\in \mathcal{O}(\epsilon)$ and 
\begin{align*}
    Z-\dif_{\pi_{\mathfrak{g}}}X=&\ Z-\dif_{\pi_{\mathfrak{g}}}h^1(Z) =  h^2([\pi_{\mathfrak{g}},Z]) =-\frac{1}{2}h^2([Z,Z])\in \mathcal{O}(\epsilon^2),
\end{align*}
where we used the homotopy relation and we have rewritten the Poisson equation for $\pi=\pi_{\mathfrak{g}}+Z$ as a Maurer-Cartan equation for $Z$:
\[ [\pi,\pi]=0 \ \ \ \ \Leftrightarrow \ \ \ \ \dif_{\pi_{\mathfrak{g}}}Z+\frac{1}{2}[Z,Z]=0.\]

This allows us to conclude that
\begin{align*}
    |\phi_X^*(\pi)-\pi_{\mathfrak{g}}|\le \mathcal{O}(\epsilon^2).
\end{align*}
We define the following algorithm 
\[ \pi_{i+1}:=\phi_{X_{i}}^*(\pi_{i}) \ \ \ \ \ \text{ with } \ \ \ X_i:=h^1(\pi_i-\pi_{\mathfrak{g}}). \]
From the discussion above we obtain
\[ |\pi_{i+1}-\pi_{\mathfrak{g}}|\le C\, \epsilon^{2^i} \ \ \ \ \ \text{ and } \ \ \ \ \ \sum_{i\geq 0} |X_i|< \infty. \]
Therefore $\Phi:= \phi_{X_0}\circ\phi_{X_1} \circ \dots$ converges to a diffeomorphism satisfying 
\[\Phi^*(\pi)=\pi_{\mathfrak{g}}.\]

To implement this algorithm for smooth flat functions, several problems need to be overcome. First, the flows will be defined only on smaller domains, so we need to decrease the size of the balls in every step in a controlled way. The second problem is that, instead of Banach spaces, we work with the Fr\'echet space of flat functions, with the family of norms $\fnorm{\cdot}_{n,k,r}$. In order to compensate for the ``loss of derivatives'' in the estimate \eqref{eq: property H estimate} of the homotopy operator, we modify the iteration using a version of Nash's smoothing operators adapted to the flat setting.

\section{Prerequisites}\label{section: prerequisits}

The purpose of this section is to establish all the tools needed for the Nash-Moser method. Namely, we give estimates for the Schouten-Nijenhuis bracket, for the flow of vector fields and for the pullback operation via flows. 

Throughout this section $B_r\subset \R^m$ denotes the open ball of radius $r$ at the origin, and $\overline{B}_r$ its closure. On the Fr\'echet space $\mathfrak{X}^{\bullet}(\overline{B}_r)$ of multivector fields on $\overline{B}_r$, we use standard $C^n$-norms $\norm{\cdot}_{n,r}$, and on the Fr\'echet space of multivector fields flat at the origin $\mathfrak{X}^{\bullet}_0(\overline{B}_r)$, we use the norms $\fnorm{\cdot}_{n,k,r}$ recalled in Appendix \ref{subsec: equivalent seminorms}. We use the notation for partial derivatives $D^a$, $a\in \N_0^m$, from Appendix B in Part I. 

A useful tool is the following trivial inequality: 
\begin{align}\label{eq: convex}
x^{\lambda}y^{1-\lambda}\le &\ x+y \ \ \ \ \text{ for }\ x,y\in [0,\infty) , \ \lambda\in [0,1]
\end{align}


\subsection{The Schouten-Nijenhuis bracket}
We begin with the estimates for the Schouten-Nijenhuis bracket of (flat) multivector fields with respect to the weighted norms:
\begin{lemma}\label{lemma: SN}
The Schouten-Nijenhuis bracket satisfies the estimate
\begin{align*}
	\fnorm{[V,W]}_{n,k+l,r}\le &\ C (\fnorm{V}_{0,k,r}\fnorm{W}_{n+1,l,r}+\fnorm{V}_{n+1,k,r}\fnorm{W}_{0,l,r})
\end{align*}
for all $V,W\in\mathfrak{X}^{\bullet}_0(\overline{B}_r)$ and $k,l,n\in \N_0 $, with constants $C=C(n,k,l,r)>0$ depending continuously on $r>0$. The inequality also holds for $W\in\mathfrak{X}^{\bullet}(\overline{B}_r)$, if we take $l=0$.
\end{lemma}
\begin{proof}
For $V\in\mathfrak{X}^{\bullet}_0(\overline{B}_r)$, $W\in\mathfrak{X}^{\bullet}(\overline{B}_r)$ and $a\in \N_0^m$ with $|a|=n$, we have
\begin{align*}
    \frac{|D^{a}[V,W](z)|}{|z|^{k+l}} \le &\ \sum_{|b_1| +|b_2|=n+1} \frac{|D^{b_1}V(z)|}{|z|^k}\frac{|D^{b_2}W(z)|}{|z|^l}.
\end{align*}
Using the interpolation inequalities in \eqref{eq: standard interpolation} and \eqref{eq: interpolation t} we have for $1\le i\le n$
\begin{align*}
\fnorm{W}_{i,l,r}\le &\  C\fnorm{W}_{0,l,r}^{\frac{n+1-i}{n+1}}\fnorm{W}_{n+1,l,r}^{\frac{i}{n+1}}\\     
\fnorm{V}_{n+1-i,k,r}\le &\ C\fnorm{V}_{n+1,k,r}^{\frac{n+1-i}{n+1}}\fnorm{V}_{0,k,r}^{\frac{i}{n+1}}
\end{align*}
Applying inequality \eqref{eq: convex}, these give the estimate from the statement
\begin{align*}
    \fnorm{[V,W]}_{n,k+l,r}\le &\ C \sum_{i=0}^{n+1}\fnorm{V}_{n+1-i,k,r}\fnorm{W}_{i,l,r}\\
    \le&\ C \sum_{i=0}^{n+1} (\fnorm{V}_{n+1,k,r}\fnorm{W}_{0,l,r})^{\frac{n+1-i}{n+1}}(\fnorm{V}_{0,k,r}\fnorm{W}_{n+1,l,r})^{\frac{i}{n+1}}\\
    \le&\ C(\fnorm{V}_{0,k,r}\fnorm{W}_{n+1,l,r} +\fnorm{V}_{n+1,k,r}\fnorm{W}_{0,l,r})\qedhere
\end{align*}
\end{proof}



\subsection{Estimates for flows in the flat setting}

In this subsection we prove estimates for the pullback of (flat) multivector fields by flows of flat vector fields with respect to the weighted norms. The proof is based on the proof of \cite[Lemmas 3.10 and 3.11]{Marcut} which provide similar estimates for the usual $C^n$-norms.

\begin{lemma}\label{lemma: flow}
There exists $\theta>0$ such that for all $0<s<r$ and all flat vector fields $Y\in \mathfrak{X}^1_0(\overline{B}_r)$ with
\[\norm{Y}_{0,r}<(r-s)\theta \quad \text{ and }\quad \norm{Y}_{1,r}<\theta \]
the flow of $Y$ is well-defined for $t\in [0,1]$ as a map
\[ \phi ^t_Y:\overline{B}_s\to B_r.\]
Moreover, the following hold for all $n,k\in\N_0$:
\begin{enumerate}[(a)]
    \item The time-one flow $\phi_Y:=\phi^1_Y$ satisfies:
\[\norm{\phi_Y-\id}_{n,s} \le  C \norm{Y}_{n,r}\]
where $C=C(n,r)$ depends continuously on $r$.
\item The pullback along $\phi_Y$ of flat multivector fields satisfies:
\begin{align*}
    \fnorm{\phi_Y^*(V)}_{n,k,s} \le &\ C (\fnorm{V}_{n,k,r}+\fnorm{V}_{0,k,r}\norm{Y}_{n+1,r} )
\end{align*}
for all $V\in \mathfrak{X}_0^{\bullet}(\overline{B}_r)$, where $C=C(k,n,r)$ depends continuously on $r$.
\item For any multivector field $W\in \mathfrak{X}^{\bullet}(\overline{B}_r)$ we have the following inequalities:
\begin{align*}
    \fnorm{\phi_Y^*(W)-W|_{\overline{B}_s}}_{n,k,s} \le \ C_W & (\fnorm{Y}_{n+1,k,r}+\fnorm{Y}_{1,k,r}\norm{Y}_{n+1,r})\\
    \le \ C_W & \fnorm{Y}_{n+1,k,r}(1+\norm{Y}_{n+1,r})\\
    \fnorm{\phi_Y^*(W)-W|_{\overline{B}_s}-\phi_Y^*([Y,W])}_{n,k,s}&\leq\\
\leq\ C_W ( \fnorm{Y}_{n+1,k,r}&\norm{Y}_{1,r} +\fnorm{Y}_{1,k,r}\norm{Y}_{n+2,r})
\end{align*}
where $C_W=C_W(k,n,r)$ depends continuously on $r$, and also on $W$.
\end{enumerate}
\end{lemma}

\begin{remark}
There are sharper versions for the estimates in item (c) in the lemma, which give a precise dependence on $W$; for example:
\begin{align*}
    \fnorm{\phi_Y^*(W)-W|_{\overline{B}_s}}_{n,k,s} \le \ C (\fnorm{Y}_{n+1,k,r}\norm{W}_{0,r}&+\fnorm{Y}_{0,k,r}\norm{W}_{n+1,r}\\
    \ \ \ \ \ +\fnorm{Y}_{1,k,r}\norm{W}_{1,r}&\norm{Y}_{n+1,r}).
\end{align*}
However, for our purposes the estimates in the lemma will be sufficient.
\end{remark}

\begin{proof}
By \cite[Lemma 3.10]{Marcut}, there exists $\theta>0$ such that item (a) holds. For such a $\theta$, let $Y$ satisfy the assumptions and write its flow as
\[ \phi_Y:= \id +g_Y:\overline{B}_s\to B_r\]

To show the estimate in (b) let $V\in \mathfrak{X}_0^{p}(\overline{B}_r)$ for some $p\in \N_0$. Write $V=V^J\partial_J$, where this sum runs over all $J=\{1\leq j_1<\ldots< j_p\leq m\}$ and $\partial_J=\partial_{j_1}\wedge \ldots \wedge \partial_{j_p}$. The pullback $\phi_Y^*(V)$ can be written as
\[ \phi_Y^*(V)|_x=V^J(x+g_Y(x))\phi^*_Y(\partial_J)|_x\]
Using Cramer's rule, there exists polynomials $Q_J^I$ such that
\[\phi^*_Y(\partial_J)|_x=
Q_J^I(\dif g_Y(x))\det (\id +\dif g_Y(x))^{-p}\partial_I|_x.\]
Hence we get the expression
\[\phi_Y^*(V)|_x=V^J(x+g_Y(x))Q_J^I(\dif g_Y(x))\det (\id +\dif g_Y(x))^{-p}\partial_I|_x\]
For $a\in \N_0^m$ with $n=|a|$ we can estimate $D^{a}\phi_Y^*(V)|_x$ by estimating the coefficients
\[ D^{a}(V^J\big(x+g_Y(x))Q_J^I(\dif g_Y(x))\det (\id +\dif g_Y(x))^{-p}\big)\]
for any $I$ and $J$. Such a term is a sum of elements of the form:
\begin{equation}\label{eq: derivative multivector}
\begin{array}{rcl}
D^{b}(V^J)(x+g_Y(x))D^{b^1}(g_Y)(x)\dots D^{b^q}(g_Y)(x)\det (\id +\dif g_Y(x))^{-\tilde{p}}
\end{array}
\end{equation}
with coefficients depending only on $a$, and with indices satisfying:
\[ 0\le q ,\ \ \ p\le \tilde{p},\ \ \ 1\le |b^i| , \ \ \ |b| +\sum_{i=1}^q (|b^i|-1)=n\]
By (a) we have $\norm{g_Y}_{1,s}\le C \cdot \theta $,
and hence, by possibly shrinking $\theta$, we have
\[ \det (\id + \dif g_Y(x))^{-1} \le 2.\] 
Next, note that by the Taylor expansion formula we get
\[|\phi_Y(x)|\le |x|C(1+\norm{g_Y}_{1,s})\le C|x|.   \]
and so for $x\ne 0$ we obtain 
\[ \frac{1}{|x|}\le \frac{C}{|\phi_Y (x)|}.\]
Hence after dividing \eqref{eq: derivative multivector} by $|x|^k$ we can bound the terms there by
\[ C \sum_{j,j_1,\dots j_q}\fnorm{V}_{j,k,r}\norm{g_Y}_{j_1+1,s}\dots \norm{g_Y}_{j_q+1,s}\]
where the indices $j\geq 0$ and $j_i\geq 0$ satisfy
\[j+j_1+\ldots +j_q=n.\]
For the terms with $q>0$ we apply first (a) and then the interpolation inequality \eqref{eq: standard interpolation} to each $\norm{Y}_{j_l+1,r}$ and obtain:
\[ \fnorm{\phi_Y^*(V)}_{n,k,s}\le C\Big(\fnorm{V}_{n,k,r} +\sum_{j=0}^{n}\fnorm{V}_{j,k,r}\norm{Y}_{n-j+1,r}\Big).\]
By applying \eqref{eq: interpolation t} to $\fnorm{V}_{j,k,r}$ and \eqref{eq: standard interpolation} to $\norm{Y}_{n-j+1,r}$ we get
\begin{equation*}
\begin{array}{rcl}
\fnorm{V}_{j,k,r}&\le &C \fnorm{V}_{0,k,r}^{1-\frac{j}{n}}\fnorm{V}_{n,k,r}^{\frac{j}{n}}\\
\norm{Y}_{n+1-j,r}&\le &C \norm{Y}_{1,r}^{\frac{j}{n}}\norm{Y}_{n+1,r}^{1-\frac{j}{n}}\ \le\ C\norm{Y}_{n+1,r}^{1-\frac{j}{n}}
\end{array}
\end{equation*}
Hence we get the estimate
\begin{equation*} 
\begin{array}{rcl}
    \fnorm{V}_{j,k,r}\norm{Y}_{n-j+1}&\le &C ((\fnorm{V}_{0,k,r}\norm{Y}_{n+1,r})^{1-\frac{j}{n}}\fnorm{V}_{n,k,r}^{\frac{j}{n}}\\
    &\stackrel{\eqref{eq: convex}}{\le} &C (\fnorm{V}_{n,k,r}+\fnorm{V}_{0,k,r}\norm{Y}_{n+1,r})
\end{array}
\end{equation*}
which implies (b).

For (c) let $W\in\mathfrak{X}^p(\overline{B}_r)$ and denote by $W_t$ the multivector field
\[ W_t:=(\phi_Y^t)^*(W)-W|_{\overline{B}_s}. \] 
Then we have $W_0 =0$, $W_1 =\phi_Y^*(W)-W|_{\overline{B}_s}$ and 
\[ \frac{\dif}{\dif t}W_t= (\phi_Y^t)^*([Y,W]).\]
Therefore we can write
\[ \phi_Y^*(W)-W|_{\overline{B}_s}=\int_0^1(\phi_Y^t)^*([Y,W])\dif t\]
Since $Y$ is flat at the origin, so is $[Y,W]$ and hence (b) implies
\begin{align*}
    \fnorm{\phi_Y^*(W)-W|_{\overline{B}_s}}_{n,k,s}\le&\ C (\fnorm{[Y,W]}_{n,k,r}+\fnorm{[Y,W]}_{0,k,r}\norm{Y}_{n+1,r})
\end{align*}
Applying Lemma \ref{lemma: SN} yields the first inequalities
\begin{align*}
    \fnorm{\phi_Y^*(W)-W|_{\overline{B}_s}}_{n,k,s}\le &\ C_W (\fnorm{Y}_{n+1,k,r}+\fnorm{Y}_{1,k,r}\norm{Y}_{n+1,r})\\
    \le &\ C_W \fnorm{Y}_{n+1,k,r}(1+\norm{Y}_{n+1,r})
\end{align*}

For the last inequality in (c) we introduce 
\[ W_t:=(\phi_Y^t)^*(W)-W|_{\overline{B}_s} -t(\phi_Y^t)^*([Y,W])\] 
Then $W_0 =0$, $W_1 =\phi_Y^*(W)-W|_{\overline{B}_s} -\phi_Y^*([Y,W])$ and 
\[ \frac{\dif}{\dif t}W_t= -t(\phi_Y^t)^*([Y,[Y,W]]).\]
implying that
\[ W_1 =-\int_0^1t(\phi_Y^t)^*([Y,[Y,W])\dif t\]
Since $Y$ is flat at the origin, so is $[Y,[Y,W]]$. Hence (b) yields
\begin{align*}
    \fnorm{W_1}_{n,k,s}\le &\ C(\fnorm{[Y,[Y,W]]}_{n,k,r}+\fnorm{[Y,[Y,W]]}_{0,k,r}\norm{Y}_{n+1,r}) 
\end{align*}
Applying Lemma \ref{lemma: SN} to the brackets gives
\begin{align*}
    \fnorm{[Y,[Y,W]]}_{n,k,r} \le &\ C_W (\fnorm{Y}_{n+1,k,r}\norm{[Y,W]}_{0,r}+\fnorm{Y}_{0,k,r}\norm{[Y,W]}_{n+1,r})\\
    \le &\ C_W (\fnorm{Y}_{n+1,k,r}\norm{Y}_{1,r}+\fnorm{Y}_{0,k,r}\norm{Y}_{n+2,r})\\
    \fnorm{[Y,[Y,W]]}_{0,k,r} \le &\ C_W \fnorm{Y}_{1,k,r}\norm{Y}_{2,r}
\end{align*}
By the standard interpolation inequality \eqref{eq: standard interpolation} we have
\[\norm{Y}_{2,r}\norm{Y}_{n+1,r}\le C \norm{Y}_{n+2,r} \]
and hence we obtain the last estimate in (c):
\begin{align*}
    \fnorm{W_1}_{n,k,s} \le &\ C_W (\fnorm{Y}_{n+1,k,r}\norm{Y}_{1,r}+\fnorm{Y}_{1,k,r}\norm{Y}_{n+2,r})\qedhere
\end{align*}
\end{proof}

\section{The algorithm}\label{section: welldefined}

The aim of this section is to set up the algorithm underlying the proof of Theorem \ref{theorem: rigidity}. The algorithm is based on the fast convergence method by Nash-Moser as used by Conn \cite{Conn85} and later by the first author in \cite{Marcut}. Throughout this section we assume the setting of Theorem \ref{theorem: rigidity}.

In the estimates below we will use $C$ to denote several different constants, arising from different estimates established so far (e.g., smoothing operator, homotopy operators, interpolation inequalities, estimates for the flow, etc.). In principle, these numbers depend continuously on the radius $\rho$ of the ball we are working on. However, since $\rho\in[r,R]$, we may ignore this dependence.

\subsection{The definition of the algorithm}
To define the algorithm we denote by $h^i_{\rho}$, for $i=1,2$ and ${\rho}\in [r,R]$ the homotopy operators from \eqref{eq: homotopy on general ball} on the flat Poisson complex in degree two. These satisfy the estimates
\begin{align}\label{eq: homotopy family estimate flat}
         \fnorm{h^i_{\rho}(W)}_{n,k,{\rho}} &\le  C\fnorm{W}_{n+a,k+bn+c,{\rho}}\nonumber \\
         &\le C\fnorm{W}_{n+p-1,k+(p-1)n,\rho}
  \end{align}
  for all $W\in\mathfrak{X}^{i+1}_0(\overline{B}_{\rho})$, for some $C=C(n,k)$, where we have denoted:
\[p:=\mathrm{max}(a+c+1,b+1,22),\]
and have used Corollary \ref{corollary: degree shift}.
We define the constants
\begin{equation}\label{eq: parameters}
    \begin{array}{ccccccc}
    \begin{matrix}
    x_a:=p \\ y_a:= p^2 \end{matrix}& \qquad &\begin{matrix} x_b:=13 p  \\y_b:= 13 p^2 \end{matrix}& \qquad &\begin{matrix} x_c:= 2p\\ y_c:= 130 p^2 \end{matrix} &\qquad &\alpha:= 50p^2.
    \end{array}
\end{equation}

For the given $0<r<R$ we define the sequences
\[r_{i}:=r+\frac{R-r}{i+1},\quad  t_{i}:=t_0^{(3/2)^i},\quad i\geq 0,\]
where $t_0>1$ will be fixed later. Note that by choice of the $r_i$ we have 
\[ r<r_{i+1}<r_i\le R, \ \ \ \ \ \text{ for all } \  i\geq 0.\]

We extend the smoothing operators $\{S_{t,s,r}\}_{0<r,t,s}$ from Subsection \ref{section:combined:smoothing} to flat vector fields, by applying them to coefficients. We define: 
\[S_i:=S_{t_i^{12p},t_i,r_i}:\mathfrak{X}^1_0(\overline{B}_{r_i})\to\mathfrak{X}^1_0(\overline{B}_{r_i}).\]
Note that $1\le t_i$ for all $i\in \N_0$. Hence by Corollary \ref{corollary: smoothing flat functions} the operators $S_i$ satisfy, for $j,k,l,n\in\N_0$ and $X\in \mathfrak{X}_0(\overline{B}_{r_i})$, the estimates
\begin{align}
    \fnorm{S_i(X)}_{n+l,k+j,r_i} \le &\ C \cdot t_i^{{12p}l+j} \fnorm{X}_{n,k+2 l,r_i}\label{eq: smoothing 1}\\
    \fnorm{(S_i-\id)(X)}_{n,k+2l,r_i} \le &\  C \cdot (t_i^{-{12p}l} \fnorm{X}_{n+l,k,r_i}+t_i^{-j}\fnorm{X}_{n,k+2l+j,r_i})\label{eq: smoothing 2}
\end{align}

For the algorithm we define the sequences 
\[ \{X_i\in\mathfrak{X}^1_0(\overline{B}_{r_i})\}_{i\geq 0}, \ \ \ \ \ \ \{\pi_i\in \mathfrak{X}^2(\overline{B}_{r_i})\}_{i\geq 0}, \]
by the following recursive procedure:
\[ \pi_0:=\pi, \ \ \ \ \ Z_i:=\pi_i-\pi_{\g}|_{\overline{B}_{r_i}}, \ \ \ \ X_i:=S_i(h^1_{r_i}(Z_i)), \ \ \ \ \ \pi_{i+1}:=\phi_{X_i}^*(\pi_i).\]
This is almost the iteration described in Subsection \ref{section: flat NM}, with the modifications that we have inserted the smoothing operators to account for the loss of derivatives, and the domains change from step to step.

We will prove by induction the following statements:

the flat bivector $Z_i\in \mathfrak{X}^2_0(\overline{B}_{r_i})$ satisfies the inequalities
\begin{align}
\fnorm{Z_i}_{x_a,y_a,r_i}&\, \le t_i^{-\alpha}\label{eq: ai}\tag{$a_i$}\\
\fnorm{Z_i}_{x_b,y_b,r_i}&\, \le t_i^{\alpha}\label{eq: bi}\tag{$b_i$}\\
\fnorm{Z_{i}}_{x_c,y_c,r_i}&\, \le t_i^{\alpha}\label{eq: ci}\tag{$c_i$},
\end{align}
and, the time-one flow of $X_i$ is defined as a map
\begin{equation}\label{eq: di}
\phi_{X_i}:\overline{B}_{r_{i+1}}\to B_{r_i}\tag{$d_i$}
\end{equation}

From here on we assume that $\pi$ satisfies, for $\bullet=a,b,c$, the inequality:
\begin{align}\label{eq: starting assumption}
\fnorm{Z_0}_{x_\bullet,y_\bullet,R}=\fnorm{\pi -\pi_{\g}|_R}_{x_\bullet,y_\bullet,R}<t_0 ^{-\alpha},
\end{align}
This follows from the assumption of Theorem \ref{theorem: rigidity} and by using Corollary \ref{corollary: degree shift}:
\[\fnorm{\pi -\pi_{\g}|_R}_{x_\bullet,y_\bullet,R}\leq
C\fnorm{\pi -\pi_{\g}|_R}_{x_\bullet+y_\bullet,0,R}\leq 
C\delta,
\]
if we choose $N=\mathrm{max}(x_a+y_a,x_b+y_b,x_c+y_c)$ and $\delta<C^{-1}t_0^{-\alpha}$. 

\subsection{The algorithm is well-defined}

The constant $t_0>1$ will be chosen such that it is larger than a finite number of quantities which will appear in the proof. 

\begin{lemma}\label{lemma: well-definedness}
There exists $t_0> 1$ such that $(a_i)-(d_i)$ hold for all $i\in \N_0$. 
\end{lemma}
\begin{proof}
Note that $(a_0) - (c_0)$ follow from \eqref{eq: starting assumption}. We show
\begin{align*}
(a_i)&\ \Rightarrow (d_i) \\
(a_i)\land (b_i)\land (c_i)\land (d_i)&\ \Rightarrow (a_{i+1})\land (b_{i+1})\land (c_{i+1}). 
\end{align*}

The following estimate will be used repeatedly:
\begin{lemma}
For all $0\leq  l \le n$ and $0
\leq j\le k+2l$, we have:
\begin{equation}\label{eq: Xk estimate}
        \fnorm{X_i}_{n,k,r_i}\le C t^{{12p}l+j}_i\fnorm{Z_i}_{n-l+p-1,k+2l-j +(p-1)(n-l),r_i}.
\end{equation}
\end{lemma}
\begin{proof} 
We use \eqref{eq: smoothing 1} and \eqref{eq: homotopy family estimate flat}: 
\begin{align*}
        \fnorm{X_i}_{n,k,r_i}&\, =\fnorm{S_i(h^1_{r_i}(Z_i))}_{n,k,r_i}  \le C t^{{12p}l+j}_i\fnorm{h^1_{r_i}(Z_i)}_{ n-l,k+2l-j,r_i}\\
        &\,\le C t^{{12p}l+j}_i\fnorm{Z_i}_{n-l+p-1,k+2l-j +(p-1)(n-l),r_i}\qedhere
\end{align*}
\end{proof}

Note that, by choosing $t_0$ sufficiently large, we have 
\[t_i^{-1}\le r_i-r_{i+1}=\frac{R-r}{(i+1)(i+2)}, \quad \textrm{for all}\ \ i\geq 0.\]
By using \eqref{eq: Xk estimate} for $l=j=0$ and \eqref{eq: ai} we get:
\begin{align}\label{eq: xo}
 \fnorm{X_i}_{1,0,r_i}&\,\le C \fnorm{Z_i}_{p,p-1,r_i}\quad \ \le \ C t^{-\alpha}_i\\
&\, \leq C t_i^{ 1-\alpha}(r_i-r_{i+1}) \quad < \ \theta \cdot (r_i-r_{i+1}),\nonumber
\end{align}
where in the last step we have used that $t_i^{1-\alpha}\leq t_0^{1-\alpha}< \theta/C$, and $\theta$ is the constant from Lemma \ref{lemma: flow}. This inequality and Lemma \ref{lemma: flow} imply \eqref{eq: di}.

$\noindent{\underline{(c_{i+1})}:}$ we use estimates (b) and (c) in Lemma \ref{lemma: flow} to obtain
 \begin{align}
    \fnorm{Z_{i+1}}_{x_c,y_c,r_{i+1}}= &\  \fnorm{\phi_{X_i}^*(Z_i)+\phi_{X_i}^*(\pi_{\g})-\pi_{\g}}_{x_c,y_c,r_{i+1}}\nonumber\\
    \le &\ C\Big(\fnorm{Z_i}_{x_c,y_c,r_i} (1+\fnorm{X_i}_{x_c+1,0,r_i})\label{eq: ek estimate}\\ 
    &\ \ \ +\fnorm{X_i}_{x_c+1,y_c,r_i}(1+\fnorm{X_i}_{x_c+1,0,r_i})\Big)\nonumber
\end{align}
Using \eqref{eq: Xk estimate} with $l=x_c$ and $j=0$ and that: 
\[p= x_a\qquad p-1+2x_c\leq y_a,\qquad {12p}x_c\leq \alpha, \]
the term $\fnorm{X_i}_{x_c+1,0,r_i}$ can be estimated by
\begin{equation}\label{eq: Xkb1 estimate}
    \fnorm{X_i}_{x_c+1,0,r_i}\ \stackrel{\eqref{eq: Xk estimate}}{\le}\ C t_i^{{12p}x_c}\fnorm{Z_i}_{x_a,y_a,r_i} \ \stackrel{\eqref{eq: ai}}{\le} \ Ct_i^{{12p}x_c-\alpha } \quad \le \quad C.
\end{equation}
Applying \eqref{eq: Xk estimate} with $l=p$ and $j=(p-1)(x_c-p+1)+2p$ we obtain
\begin{equation}\label{eq: k mixed}
        \fnorm{X_i}_{x_c+1,y_c,r_i} \stackrel{\eqref{eq: Xk estimate}}{\le}  Ct_i^{({12p}+2)p+(p-1)(x_c-p+1)}\fnorm{Z_i}_{x_c,y_c,r_i} 
\end{equation}
Hence we get from \eqref{eq: ek estimate} the estimate 
\begin{align*}
    \fnorm{Z_{i+1}}_{x_c,y_c,r_{i+1}}\quad \le \quad &  C(1+t_i^{({12p}+2)p+(p-1)(x_c-p+1)})\fnorm{Z_i}_{x_c,y_c,r_i}\\
    \stackrel{\eqref{eq: ci}}{\le}\ \ &  Ct_i^{13p^2+2p-1+\alpha}\ \ \le \ \  Ct_i^{-\epsilon} t_{i+1}^{\alpha} \ \ \le \ \ Ct_0^{-\epsilon }t_{i+1}^{\alpha}
\end{align*}
where we have used that: 
\begin{equation}\label{eq: random ie2}
   13p^2+2p-1< \frac{1}{2}\alpha,
\end{equation}      
and we have denoted by $\epsilon>0$ a number smaller than the difference. By taking $t_0$ large, enough, we can make $Ct_0^{-\epsilon}<1$, and so $(c_{i+1})$ holds. We will use tacitly a similar argumentation later on.\\

$\noindent{\underline{(b_{i+1})}:}$ we use again estimates (b) and (c) in Lemma \ref{lemma: flow} to obtain
\begin{align}
    \fnorm{Z_{i+1}}_{x_b,y_b,r_{i+1}}= &\ \fnorm{\phi_{X_i}^*(Z_i)+\phi_{X_i}^*(\pi_{\g})-\pi_{\g}}_{x_b,y_b,r_{i+1}}\nonumber\\
    \le &\ C\big(\fnorm{Z_i}_{x_b,y_b,r_i}+\fnorm{X_i}_{x_b+1,y_b,r_i}\label{eq: ck estimate}\\ 
    &\ \ +(\fnorm{Z_i}_{0,y_b,r_i}+\fnorm{X_i}_{1,y_b,r_i})\fnorm{X_i}_{x_b+1,0,r_i}\big)\nonumber
\end{align}
For the first term in \eqref{eq: ck estimate} we immediately get
\begin{equation*}
    \fnorm{Z_i}_{x_b,y_b,r_i}\stackrel{\eqref{eq: bi}}{\le} Ct^{\alpha}_i \quad \le \quad Ct_0^{-\epsilon}t^{\alpha}_{i+1}
\end{equation*}
The second term is estimated exactly as \eqref{eq: k mixed} and by using \eqref{eq: bi}:
\begin{align*}
 \fnorm{X_i}_{x_b+1,y_b,r_i}  \stackrel{\eqref{eq: Xk estimate}}{\le}  & Ct_i^{
 ({12p}+2)p+(p-1)(x_b-p+1)}\fnorm{Z_i}_{x_b,y_b,r_i}  \\
    \stackrel{\eqref{eq: bi}}{\le } \ & 
     Ct_i^{24p^2-9p-1+\alpha}\le \ C t_0^{-\epsilon} t^{\alpha }_{i+1} 
\end{align*}
where we have used that:
\begin{equation*}
         24p^2-9p-1< \frac{1}{2}\alpha .
\end{equation*}
We move to the last term. Using \eqref{eq: Xk estimate} with $l=p-1$ and $j=0$ we obtain:
\begin{equation}\label{eq: Xkq1 estimate}
    \fnorm{X_i}_{x_b+1,0,r_i}\stackrel{\eqref{eq: Xk estimate}}{\le} Ct_i^{12p^2} \fnorm{Z_i}_{x_b,2p+(p-1)(x_b-p+1),r_i} \stackrel{\eqref{eq: bi}}{\le}  Ct_i^{{12p^2}+\alpha} 
\end{equation}
where we used that \[2p+(p-1)(x_b-p+1)\le y_b.\] 

Using \eqref{eq: Xk estimate} for $l=j=0$ and interpolation in the weights, we obtain
\begin{align}\label{eq: X1 estimate}
    \fnorm{X_i}_{1,y_b,r_i} \quad \stackrel{\eqref{eq: Xk estimate}}{\le} \quad & C\fnorm{Z_i}_{p,y_b+p-1,r_i}\nonumber \\
    \stackrel{\eqref{eq: interpolation s}}{\le}\quad & C \fnorm{Z_i}_{p,y_a,r_i}^{\frac{y_c-y_b-p+1}{y_c-y_a}}\fnorm{Z_i}_{p,y_c,r_i}^{\frac{y_b+p-1-y_a}{y_c-y_a}}\\
    \stackrel{\eqref{eq: ai},\eqref{eq: ci}}{\le}&Ct_i^{-\alpha \frac{y_c+y_a-2(y_b+p-1)}{y_c-y_a}}\nonumber
\end{align}
The above also yields:
\begin{align}
    \fnorm{Z_i}_{0,y_b,r_i}\quad \le\quad &  \fnorm{Z_i}_{p,y_b+p-1,r_i} \le \ C t_{i}^{-\alpha \frac{y_c+y_a-2(y_b+p-1)}{y_c-y_a}} \label{eq: z1q2}
\end{align}
Combining \eqref{eq: Xkq1 estimate}, \eqref{eq: X1 estimate} and \eqref{eq: z1q2}, the last term in \eqref{eq: ck estimate} is bounded by
\begin{align*}
    C t_i^{{12p^2}+2\alpha \frac{y_b-y_a+p-1}{y_c-y_a}} \ \le \ Ct_0^{-\epsilon} t_{i+1}^{\alpha}
\end{align*}
where we have used that 
\[12p^2+2\alpha \frac{y_b-y_a+p-1}{y_c-y_a}< \frac{3}{2}\alpha.\]
Adding up these inequalities, we obtain that $(b_{i+1})$ holds.\\

$\noindent{\underline{(a_{i+1})}:}$ we write $Z_{i+1}=V_i +\phi_{X_i}^*(U_i)$, where
\begin{equation}\label{eq: split Z}
    V_i:= \phi_{X_i}^*(\pi_{\g})-\pi_{\g}-\phi_{X_i}^*([X_i,\pi_{\g}]), \quad \text{ and }\quad U_i:=Z_i -[\pi_{\g},X_i].
\end{equation}
Applying (c) in Lemma \ref{lemma: flow} we obtain for $V_i$ the estimate
\begin{align}\label{eq: vi estimate}
    \fnorm{V_i}_{x_a,y_a,r_{i+1}} \ \le \ C( \fnorm{X_i}_{x_a+1,y_a,r_i}\fnorm{X_i}_{1,0,r_i}+\fnorm{X_i}_{1,y_a,r_i}\fnorm{X_i}_{x_a+2,0,r_i})
\end{align}
Applying \eqref{eq: Xk estimate} with $l=p$ and $j=3p-1$ yields the estimate
\begin{align}\label{eq: Xsas}
    \fnorm{X_i}_{x_a+1,y_a,r_i}&\ \stackrel{\eqref{eq: Xk estimate}}{\le} Ct^{12p^2+3p-1}_i\fnorm{Z_i}_{x_a,y_a,r_i}\nonumber \\
    &\ \ \stackrel{\eqref{eq: ai}}{\le}\ Ct^{12p^2+3p-1-\alpha}_i.
\end{align}
From \eqref{eq: Xk estimate} for $l=0$ and $j=p-1$ we obtain
\begin{equation}\label{eq: Xsasb}
\fnorm{X_i}_{1,y_a,r_i}\stackrel{\eqref{eq: Xk estimate}}{\le} Ct^{p-1}_i\fnorm{Z_i}_{x_a,y_a,r_i}\stackrel{\eqref{eq: ai}}{\le}Ct^{p-1-\alpha}_i
\end{equation}
Using \eqref{eq: Xk estimate} with $l=p+1$ and $j=0$ provides for $\fnorm{X_i}_{x_a+2,0,r_i}$ the estimate
\begin{equation}\label{eq: Xs0}
    \fnorm{X_i}_{x_a+2,0,r_i} \stackrel{\eqref{eq: Xk estimate}}{\le} Ct^{12p(p+1)}_i\fnorm{Z_i}_{x_a,3p+1,r_i} \stackrel{\eqref{eq: ai}}{\le}Ct^{{12p}(p+1)-\alpha}_i
\end{equation}
where we used that $3p+1\leq y_a$. 

Combining \eqref{eq: vi estimate}, \eqref{eq: xo} and \eqref{eq: Xsas} - \eqref{eq: Xs0} we obtain for $V_i$ the estimate
\begin{align}\label{eq: Vk}
    \fnorm{V_i}_{x_a,y_a,r_{i+1}} \ &\le \ C (t^{12p^2+3p-1-2\alpha}_i+
    t^{{12p}(p+1)+p-1-2\alpha}_i
    )\nonumber\\ 
    &\le \ C t_i^{12p^2+13p-1-2\alpha}\ \le \  Ct^{-\epsilon}_0 t^{-\alpha }_{i+1} 
\end{align}
where we have used:
\begin{equation}\label{eq: not very strict inequality}
12p^2 +13p-1 < \frac{1}{2} \alpha.
\end{equation}
 
For $\phi_{X_i}^*(U_i)$ we use Lemma \ref{lemma: flow} (b), \eqref{eq: Xs0} and $12p(p+1)\leq \alpha$ to obtain
\begin{align}\label{eq: Uk estimate}
    \fnorm{\phi^*_{X_i}(U_i)}_{x_a,y_a,r_{i+1}}\le &\ C\fnorm{U_i}_{x_a,y_a,r_i}(1+\fnorm{X_i}_{x_a+1,0,r_i})
    \nonumber \\ 
    \le& \ C\fnorm{U_i}_{x_a,y_a,r_i}
\end{align}
To estimate this term, we write $U_i$ as
\begin{align}\label{eq: uk representation}
         U_i& \,= Z_i-[\pi_{\g},X_i] 
        = [\pi_{\g},(\id-S_i)(h^1_{r_i}(Z_i))]-\frac{1}{2}h^2_{r_i}([Z_i,Z_i])
\end{align}
where we used the homotopy identity and that $\pi_i$ is Poisson, i.e.
\begin{align*}
    \dif_{\pi_{\g}}Z_i +\frac{1}{2}[Z_i,Z_i]=0.
\end{align*}
For the first term we apply the Lemma \ref{lemma: SN} to get
\begin{equation}\label{eq: uk1}
    \fnorm{[\pi_{\g},(\id-S_i)(h^1_{r_i}(Z_i))]}_{x_a,y_a,r_i} \le C \, \fnorm{(\id-S_i)(h^1_{r_i}(Z_i))}_{x_a+1,y_a,r_i}
\end{equation}
Using the second smoothing inequality \eqref{eq: smoothing 2} with the parameters
\[  l:=11p \qquad \text{ and }\qquad  j:=y_c-2p^2\geq 0\]
and the bound for the homotopy operators \eqref{eq: homotopy family estimate flat}, we get: 
\begin{align}
    \fnorm{(\id-S_i)(h^1_{r_i}(Z_i))}&_{x_a+1,y_a,r_i} \nonumber\\
    \le  C \big(&t_i^{-132p^2}\, \fnorm{h^1_{r_i}(Z_i)}_{x_b-p+1,y_a-22p,r_i}\nonumber\\
     +&\, t_i^{-y_c+2p^2}\fnorm{h^1_{r_i}(Z_i)}_{x_a+1,y_c-p^2,r_i}\big)\nonumber \\
     \le  C \big(&t_i^{-132p^2}\, \fnorm{Z_i}_{x_b,y_a +(p-1)(x_b-p+1)-22p,r_i}\nonumber\\
     +&\, t_i^{-y_c+2p^2}\fnorm{Z_i}_{x_a+p,y_c-1,r_i}\big)\label{eq: uk1e}\\
    \le   C( t_i^{-132p^2}&\, \fnorm{Z_i}_{x_b,y_b,r_i} +t_i^{-y_c+2p^2}\fnorm{Z_i}_{x_c,y_c,r_i})\nonumber\\
    \le  C( t_i^{-132p^2}&\,+t_i^{-y_c+2p^2})t_i^{\alpha} \ \le \ Ct_0^{-\epsilon}t_{i+1}^{-\alpha}\nonumber
\end{align}
where we used \eqref{eq: bi}, \eqref{eq: ci}, the relations 
\begin{align*}
    22p \le y_a, \qquad  x_a+p = x_c,\qquad
    y_a +(p-1)(x_b-p+1)-22p\le y_b, 
\end{align*}
and that we have
\begin{align*}
    -132p^2&\ < \ -\frac{5}{2}\alpha, \qquad \text{ and } \qquad \
    -y_c+2p^2\ < \  -\frac{5}{2}\alpha. 
\end{align*}
For the second term in \eqref{eq: uk representation} we estimate the Schouten-Nijenhuis bracket similarly as in the proof of Lemma \ref{lemma: SN} to obtain
\begin{align}\label{eq: uk2}
        \fnorm{\frac{1}{2}h^2_{r_i}([Z_i,Z_i])}_{x_a,y_a,r_i}&\, \le C \fnorm{[Z_i,Z_i]}_{x_a+p-1,y_a+(p-1)x_a,r_i} \nonumber\\
        &\, \le  C \sum_{j=0}^{x_a}\fnorm{Z_i}_{x_a-j,y_a,r_i} \fnorm{Z_i}_{p+j,(p-1)x_a,r_i}\nonumber\\
        &\, \le  C \fnorm{Z_i}_{x_a,y_a,r_i} \fnorm{Z_i}_{p+x_a,y_a,r_i}  \\
        &\, \le Ct_i^{-\alpha}\fnorm{Z_i}_{x_a,y_a,r_i}^{1-\frac{p}{x_b-x_a}}\fnorm{Z_i}_{x_b,y_a,r_i}^{\frac{p}{x_b-x_a}}\nonumber\\
        &\, \le 
        C t_i^{-2\alpha(1-\frac{p}{x_b-x_a})}\, \le 
        Ct_0^{-\epsilon}t_{i+1}^{-\alpha}\nonumber
\end{align}
where we used the interpolation \eqref{eq: interpolation t} and the relations
\[
  (p-1)x_a\le y_a\le y_b, \qquad
4p< (x_b-x_a). 
\]
Finally, using \eqref{eq: Xs0}, \eqref{eq: uk1e} and \eqref{eq: uk2} we can estimate \eqref{eq: Uk estimate} by
\begin{align*}
    \fnorm{\phi^*_{X_i}(U_i)}_{x_a,y_a,r_{i+1}}\le &\ C\fnorm{U_i}_{x_a,y_a,r_i}\   \le \ C t^{-\epsilon}_0t_{i+1}^{-\alpha}
\end{align*}
This estimate together with \eqref{eq: Vk} proves $(a_{i+1})$ and hence Lemma \ref{lemma: well-definedness}.
\end{proof}

\begin{remark}
The constant $t_0>1$ will be chosen according to Lemma \ref{lemma: well-definedness}, and such that it satisfies two more estimates, which will be explained in Subsection \ref{subsection:Convergence}.
\end{remark}

\section{Proof of Theorem \ref{theorem: rigidity}}\label{section: convergence}
The goal of this section is to complete the proof of Theorem \ref{theorem: rigidity}. In order to achieve this we use methods developed in \cite{LoZe} to show the the sequence 
\[ \Phi_i :=\phi_{X_0} \circ \dots \circ \phi_{X_i}|_{\overline{B}_r} : \overline{B}_r \to B_R \]
converges to local diffeomorphism which is moreover Poisson map.

\subsection{Improved estimates}
We assume that $\pi$ satisfies \eqref{eq: starting assumption}. An important difference to the previous section is that the various constants $C>0$ may depend on $\pi$. In the following we prove versions of the estimates \eqref{eq: ai}-\eqref{eq: ci} for general degrees. We start with the one of type \eqref{eq: ci}.

\begin{lemma}\label{lemma: ekc}
For all $y\in \N_0$ there exists $C=C_{y}(\pi)>0$ such that 
\begin{align}\label{eq: ekc}
\fnorm{Z_i}_{x_c,y,r_i}\, \le\, C\,  t_{i}^{\alpha}
\end{align}
for all $i\in \N_0$.
\end{lemma}
\begin{proof}
Note that for $i=0$, there is nothing to prove. We now use an iteration in $i$ and basically redo the estimate for \eqref{eq: ci} in the proof of Lemma \ref{lemma: well-definedness}. Similar as in \eqref{eq: ek estimate} we apply the estimates (b) and (c) in Lemma \ref{lemma: flow} and \eqref{eq: Xkb1 estimate} to obtain
\begin{align*}
       \fnorm{Z_{i+1}}_{x_c,y,r_{i+1}}\le      C(\fnorm{Z_i}_{x_c,y,r_i} +\fnorm{X_i}_{x_c+1,y,r_i}).
     \end{align*}
For the second term we use that equation \eqref{eq: k mixed} holds for general $y$, and we obtain the recursive formula:
\begin{equation*}
\fnorm{Z_{i+1}}_{x_c,y,r_{i+1}}\le C\, t_i^{s}\fnorm{Z_{i}}_{x_c,y,r_{i}},
\end{equation*}
where $s:=11p^2+4p+1+(p-1)x_c$. Iterating this inequality, we obtain:
\begin{align*}
\fnorm{Z_{i+1}}_{x_c,y,r_{i+1}}&\le C^{i+1}(t_0\cdot\ldots \cdot t_i)^{s}\fnorm{Z_{0}}_{x_c,y,r_{0}}\\
&= C^{i+1}t_0^{s(1+\frac{3}{2}+\ldots +\frac{3^{i}}{2^i})}\fnorm{Z_{0}}_{x_c,y,r_{0}}\\
&\leq C^{i+1}t_{i+1}^{2s}\fnorm{Z_{0}}_{x_c,y,r_{0}}\\
& \leq (C^{i+1}t_{i+1}^{-\epsilon})t_{i+1}^{\alpha}\fnorm{Z_{0}}_{x_c,y,r_{0}}\\
& \leq Ct_{i+1}^{\alpha}\fnorm{Z_{0}}_{x_c,y,r_{0}}.
\end{align*}
where we have used that, by \eqref{eq: random ie2} $s<\frac{1}{2}\alpha$ and that $C^{i+1}t_{i+1}^{-\epsilon}$ can be bounded independently of $i$.
\end{proof}

In all degrees, we have:

\begin{lemma}\label{lemma: ckc}
For all $x,y \in \N_0$ there exists $C=C_{x,y}(\pi)>0$ such that
\begin{align}\label{eq: ckc}
    \fnorm{Z_i}_{x,y,r_i}\le C t_{i}^{\alpha+ 2(x-x_c)(p-1)}
\end{align}
for all $i\in \N_0$.
\end{lemma}
\begin{proof}
The proof is similar to the one of the previous lemma. Using estimates (b) and (c) in Lemma \ref{lemma: flow} we get
\begin{align}\label{eq: ckc estimate}
\fnorm{Z_{i+1}}_{x,y,r_{i+1}}\le &\ C(\fnorm{Z_i}_{x,y,r_i}+\fnorm{Z_i}_{0,y,r_i}\fnorm{X_i}_{x+1,0,r_i}\nonumber\\ 
    &\ \ \ +\fnorm{X_i}_{x+1,y,r_i}+\fnorm{X_i}_{1,y,r_i}\fnorm{X_i}_{x+1,0,r_i})
     \\
 \le &\ C(\fnorm{Z_i}_{x,y,r_i}+\fnorm{X_i}_{x+1,y,r_i}(1+\fnorm{Z_i}_{0,y,r_i}+\fnorm{X_i}_{1,y,r_i})).\nonumber
\end{align}
As in \eqref{eq: k mixed}, we estimate the term $\fnorm{X_i}_{x+1,y,r_i}$ by: 
\begin{equation*}
\fnorm{X_i}_{x+1,y,r_i} \le \ Ct_i^{11p^2+4p+1+(p-1)x}\fnorm{Z_i}_{x,y,r_i}.
\end{equation*}
Next, we have that: 
\begin{align*}        \fnorm{X_i}_{1,y,r_i} & \stackrel{\eqref{eq: Xk estimate}}{\le}
C\fnorm{Z_i}_{p,y+p+1,r_i} \\
        & \stackrel{\eqref{eq: interpolation s}}{\le}  C \fnorm{Z_i}_{x_a,y_a,r_i}^{\frac{y+y_a-p+1}{2y}}\fnorm{Z_i}_{x_a,2y+y_a,r_i}^{\frac{y+p-1-y_a}{2y}}\\
        &\stackrel{\eqref{eq: ai},\eqref{eq: ekc}}{\le}  C t_i^{-\alpha\frac{y_a-p+1}{y}}\leq C,
\end{align*}
where we have assumed that $y+p-1\geq y_a$ -- otherwise apply \eqref{eq: ai} as the second step. The above also implies:
\begin{align*}        \fnorm{Z_i}_{0,y,r_i} \leq C.
\end{align*}
Using the last three inequalities in \eqref{eq: ckc estimate}, we obtain:
\begin{equation*}
\fnorm{Z_{i+1}}_{x,y,r_{i+1}} \le \ Ct_i^{s+(x-x_c)(p-1)}\fnorm{Z_i}_{x,y,r_i},
\end{equation*}
where, $s=11p^2+4p+1+(p-1)x_c$. As in the proof of the previous lemma, we iterate this relation, and we obtain the result. 
\end{proof}

We give now a new version of \eqref{eq: ai}:
\begin{lemma}\label{lemma: bkc}
For all $N \in \N_0 $ there exists $C=C_N(\pi)>0$ such that 
\begin{align}
 \fnorm{Z_i}_{x_a,y_a,r_i}\le&\  C t_{i}^{-N}\label{eq: bkc}
\end{align}
for all $i\in \N_0$.
\end{lemma}

\begin{proof}
By \eqref{eq: ai} it holds for $N=\alpha$. We show that the estimate for $N\geq \alpha$ implies the estimate for $N+1$. For this, we go again through the estimates in the proof of \eqref{eq: ai}. So we write $Z_{i+1}=V_i +\phi_{X_i}^*(U_i)$ as defined in \eqref{eq: split Z}. To estimate $V_i$
we replace \eqref{eq: ai} by inequality \eqref{eq: bkc}
in the proof of inequalities \eqref{eq: xo}, 
\eqref{eq: Xs0}, \eqref{eq: Xsas} and \eqref{eq: Xsasb}, and then obtain:
\begin{align*}
 \fnorm{X_i}_{1,0,r_i}&\,\le C t^{-N}_i\\
    \fnorm{X_i}_{x_a+2,0,r_i}&\, \le Ct^{{12p}(p+1)-N}_i\\ 
\fnorm{X_i}_{x_a+1,y_a,r_i}&\, \le Ct^{12p^2+3p-1-N}_i\\
\fnorm{X_i}_{1,y_a,r_i}&\, \le Ct^{p-1-N}_i
\end{align*}
Hence we can estimate $V_i$ using \eqref{eq: vi estimate} to obtain: 
\begin{equation}\label{eq: vi final}
\fnorm{V_i}_{x_a,y_a,r_{i+1}} \ \le \  \ C t_i^{12p^2+13p-1-2N}\ \le  C t_{i+1}^{-(N+1)},
\end{equation}
where we used the fact that 
\[ 12p^2+13p-1-2N\leq -\frac{3}{2}(N+1),\]
which holds for $N=\alpha$, and so it holds also for $N\geq \alpha$.

To estimate $\phi_{X_i}^*(U_i)$, we use \eqref{eq: Uk estimate}, we write $U_i$ as the sum \eqref{eq: uk representation}, and estimate each term separately. For the first term, we apply \eqref{eq: uk1}, then the smoothing inequality in \eqref{eq: smoothing 2}, and then homotopy inequality \eqref{eq: homotopy family estimate flat}, and, denoting $u:=x_a+p$ and $v:=y_a+(p-1)(x_a+1)$, we obtain: 
\begin{align*}
    \fnorm{(\id-S_i)(h^1_{r_i}(Z_i))}_{x_a+1,y_a,r_i} 
    \le   C \big(&t_i^{-{12p}l}\, \fnorm{Z_i}_{u+l,v+(p-3)l,r_i}+t_i^{-j}\fnorm{Z_i}_{u,v+j,r_i}\big)\\
    \stackrel{\eqref{eq: ckc}}{\le}   C \big(& t_i^{-{12p}l+\alpha+2(p-1)(u+l-x_c)} + t_i^{-j+\alpha+2(p-1)(u-x_c)}\big)\\
       =   C \big(& t_i^{-(10p+2)l+e} + t_i^{-j+f}\big),
\end{align*}
for some constants $e,f$. Since $l,j\geq 0$ are arbitrary, we can take them large enough so that:
\begin{equation}\label{eq: u1 final}
    \fnorm{(\id-S_i)(h^1_{r_i}(Z_i))}_{x_a+1,y_a,r_i}\leq  Ct_{i+1}^{-(N+1)}.
\end{equation}  

To estimate the other term of $U_i$, we use \eqref{eq: uk2}, the interpolation inequality \eqref{eq: interpolation t}, equation \eqref{eq: ckc} and the induction assumption:
\begin{align*}
        \fnorm{\frac{1}{2}h^2_{r_i}([Z_i,Z_i])}_{x_a,y_a,r_i}&\,  \le  C \fnorm{Z_i}_{x_a,y_a,r_i} \fnorm{Z_i}_{p+x_a,y_a,r_i}  \\
        &\, \le Ct_i^{-N}\fnorm{Z_i}_{x_a,y_a,r_i}^{\frac{x-x_a-p}{x-x_a}}\fnorm{Z_i}_{x,y_a,r_i}^{\frac{p}{x-x_a}}\\
        &\, \le Ct_i^{-N(2-\frac{p}{x-x_a})+(\alpha+2(p-1)(x-x_c))\frac{p}{x-x_a}}
\end{align*}

We claim that, for large $x$ the exponent of the right-hand side is smaller than $-(N+1)\frac{3}{2}$. Note that, the limit at $x\to \infty$ of the exponent is:
\[-2N+2p(p-1).\]
So it suffices to note that:
\[4p(p-1)+3< N,\]
which holds because $N\geq\alpha$. Hence
\begin{align}
    \fnorm{\frac{1}{2}h^2_{r_i}([Z_i,Z_i])}_{x_a,y_a,r_i} \le \ Ct_{i+1}^{-(N+1)} \label{eq: uk2c final}
\end{align}
Finally, combining \eqref{eq: Uk estimate}, \eqref{eq: u1 final} and \eqref{eq: uk2c final} we obtain
\begin{align*}
    \fnorm{\phi^*_{X_i}(U_i)}_{x_a,y_a,r_{i+1}}\le &\ C\fnorm{U_i}_{x_a,y_a,r_i} \le Ct_{i+1}^{-(N+1)}
\end{align*}
which finishes the proof.
\end{proof}

\subsection{Convergence}\label{subsection:Convergence}

To show convergence of the sequence we constructed, we use the following: 
\begin{lemma}\label{lemma: sigman}
For every $n\in \N$ the sum
\[ \sum_{i\geq 0}\norm{X_i}_{n,r_i} <\infty\]
Moreover, for $n=1$, we have
\[ \sum_{i\geq 0}\norm{X_i}_{1,r_i}\le C\cdot t_0^{-1} \]
\end{lemma}
\begin{proof}
The statement for $n=1$ follows directly from \eqref{eq: xo}. Note that by \eqref{eq: Xk estimate} and Corollary \ref{corollary: degree shift}, we have that:
\[  \norm{X_i }_{n,r_i}\ \le \ C \fnorm{Z_i}_{n+p-1,(p-1)n,r_i}\ \le \ C\norm{Z_i}_{pn+p-1,r_i}.
\]
So it suffices to show that $\sum_{i}\norm{Z_i}_{n,r_i}<\infty$, for all $n$. 
\begin{align*}
    \norm{Z_i}_{n,r_i}   &\ \ \ \stackrel{\eqref{eq: interpolation t}}{\le} \ \ \ C \norm{Z_i}_{x_a,r_i}^{\frac{x_a}{n}} \norm{Z_i}_{x_a+n,r_i}^{\frac{n-x_a}{n}}\\
    &\stackrel{\eqref{eq: ckc},\eqref{eq: bkc}}{\le} C  t_i^{-N\frac{x_a}{n}+\frac{n-x_a}{n}(\alpha +2(p-1)(x_a+n-x_c))},
\end{align*}
and so by choosing $N$ large enough, we obtain the conclusion.
\end{proof}

Next, we recall some facts and estimates for diffeomorphisms (for proofs see \cite[Section 3.2]{Marcut}).

\begin{lemma}\label{lemma: diffeo}
There exists $\theta>0$ such that, if $\Phi \in C^{\infty}(\overline{B}_r,\R^m)$ satisfies $\norm{\Phi-\id}_{1,r} <\theta$, then $\Phi $ is a diffeomorphism onto its image.
\end{lemma}
\begin{lemma}\label{lemma: converging diffeo}
There exists $\theta>0$ such that for all sequences of maps
\[ \{\phi_i\in C^{\infty}(\overline{B}_{r_{i+1}},B_{r_{i}})\}_{i\geq 0}\]
with $0<r\le r_{i+1} \le r_{i}<R$, which satisfy the estimates
\[ \sigma_0 := \sum_{ i\geq 0}\norm{\phi_i-\id}_{0,r_{i+1}}<\theta, \ \ \ \ \text{ and } \ \ \ \ \sigma_n:=\sum_{ i\geq 0}\norm{\phi_i -\id}_{n,r_{i+1}}<\infty,\]
the sequence of maps 
\[ \Phi_i:=\phi_0\circ \phi_1 \circ \dots \circ \phi_i |_{\overline{B}_r} :\overline{B}_r\to B_R,\]
converges in all $C^n$-norms to a smooth map $\Phi\in C^{\infty}(\overline{B}_r,B_R)$ which satisfies
\[ \norm{\Phi-\id}_{n,r}\le C \sigma_{n}e^{C\sigma_n},\]
for some constant $0<C=C_n$.
\end{lemma}

Lemma \ref{lemma: flow} and 
Lemma \ref{lemma: sigman}
imply that our vector fields satisfy:
\begin{align*}
\sum_{i\geq 0}\norm{\phi_{X_i}-\mathrm{Id}}_{1,r_{i+1}}&\leq\, C
\sum_{i\geq 0}\norm{X_i}_{1,r_{i+1}}\leq C\cdot t^{-1}_0\\
\sum_{i\geq 0}\norm{\phi_{X_i}-\mathrm{Id}}_{n,r_{i+1}}&\leq\, C
\sum_{i\geq 0}\norm{X_i}_{n,r_{i+1}}< \infty.
\end{align*}
Note that the constant in the first inequality is independent of $\pi$, as it comes from  \eqref{eq: xo} and Lemma \ref{lemma: flow}. Therefore, we can make $t_0$ large enough, so that also the first condition of Lemma \ref{lemma: converging diffeo} holds: $\sigma_0<\theta$. Hence \[ \Phi_i:=\phi_{X_0} \circ \dots \circ \phi_{X_i}|_{\overline{B}_r}:\overline{B}_r\to B_R\]
converges uniformly in all $C^n$-norms to a map $\Phi:\overline{B}_r \to B_R$ which satisfies 
\[ \norm{\Phi-\id}_{1,r}\le Ct_0^{-1}.\]
Again, this constant is independent of $\pi$, so by possibly increasing $t_0$, Lemma \ref{lemma: diffeo} implies that $\Phi$ is a diffeomorphism onto its image.

The convergence of $\Phi_i$ to $\Phi$ in the $C^1$-topology implies that $\Phi^*_i(\pi)$ convergence in the $C^0$-topology to $\Phi^*(\pi)$. As
\[Z_i|_{\overline{B}_r}=\Phi^*_i(\pi)-\pi_{\g}|_{\overline{B}_r}\xrightarrow{i\to \infty}0\] in the $C^0$-topology we conclude that $\Phi^*(\pi)=\pi_{\g}|_{\overline{B}_r}$. Hence $\Phi$ is a Poisson map and a diffeomorphism onto its image: 
\[ \Phi:(\overline{B}_r,\pi_{\g})\to (B_R,\pi)\]
This concludes the proof of Theorem \ref{theorem: rigidity}.

%% file: Setting_NM.tex
\section{Smoothing operators for flat functions}\label{section: functional analysis}

In this appendix we collect general properties on the space of flat functions on a closed ball $\overline{B}_r$, for which we develop the necessary tools to apply the Nash-Moser method in the flat framework. Most importantly, we obtain smoothing operators for this space with respect to the norms $\fnorm{\cdot}_{n,k,r}$. These operators depend on two parameters $t,s$, corresponding to the two indexes $n,k$ of the norms. For their construction, we observe that flatness at the corresponds under inversion in the sphere to Schwartz-like behavior at infinity. This allows us to use the original construction due to Nash \cite{Nash} of smoothing operators on the space of Schwartz functions. Finally, we prove interpolation inequalities for the $\fnorm{\cdot}_{n,k,r}$ norms.

\subsection{The space of flat functions}\label{subsec: equivalent seminorms}
Let $B_r\subset \R^m$ be the open ball of radius $0<r$ centered at the origin. On the space of smooth functions $C^{\infty}(\overline{B}_r)$ we have the usual $C^n$-norms:
\[     \norm{f}_{n,r}:= \sup_{x\in \overline{B}_r}\ \sup_{a\in \N_0^m:\ |a|\le n}|D^af(x)|.\]
Denote the space of functions on $\overline{B}_r$ that are flat at the origin by:
\[C^{\infty}_0(\overline{B}_r):=\{f\in C^{\infty}(\overline{B}_r)\, |\,  \forall\, a \in \N_0^m : D^{a}f(0)=0\}.\]
Flat functions have the following behavior at the origin:
\[\forall\, a\in \N^m_0,\, k\in \N_0\ :\ \lim_{|x| \to 0}|x|^{-k}|D^{a}f(x)|=0.\]
Therefore we may define for flat functions the norms:
\begin{align}\label{defintion:norms:appendix:flat}
     \fnorm{f}_{n,k,r}:= \sup_{x\in \overline{B}_r}\ \sup_{a\in \N_0^m:\ |a|\le n}|x|^{-k}|D^af(x)|
\end{align}
for $k,n\in \N_0$ and $f\in C^{\infty}_0(\overline{B}_r)$. We obtain the Fr\'echet space: 
\[\big(C^{\infty}_0(\overline{B}_r),\{ \fnorm{\cdot}_{n,k,r}\}_{k,n\in \N_0}\big).\]
This space carries also the norms introduced in \eqref{eq: flat norms 2}: 
\begin{align*}
    |f|_{n,k,r}:=\norm{\frac{1}{|x|^{k}}f}_{n,\overline{B}_r}=\sup_{x\in \overline{B}_r}\sup_{a\in \N_0^m : |a|\le n}|D^a|x|^{-k}f(x)| 
\end{align*}
The following result is very useful:
\begin{lemma}\label{lemma: equivalence}
The norms $|\cdot |_{n,k,r}$ and $\fnorm{\cdot}_{n,k,r}$ are equivalent:
\[ C^{-1}|f |_{n,k,r} \le \fnorm{f}_{n,k,r}\le C|f |_{n,k,r}\]
for a constant $0<C=C(k,n)$.
\end{lemma}
\begin{proof}
Using Taylor's formula with integral remainder of order $d$ for $f\in C^{\infty}_0(\overline{B}_r)$ and that $f$ is flat, we obtain:
\begin{align}\label{eq: taylor}
    |f(x)|\le &\ C \sum_{b\in \N_0^m:\, |b|=d}|x|^{d}\int_{0}^1(1-s)^{d-1}|D^{b} f(sx)|\dif s.
\end{align}
Using this we obtain, for all $k,l,n\in \N_0$, the inequality
\begin{align}\label{ineq:stupid}
    \fnorm{f}_{n,k+l,r}\le C\fnorm{f}_{n+l,k,r}.
\end{align}
Using the Leibniz rule, we may write for $a\in \N_0^m$ with $|a|=n$
\begin{align}\label{ineq:intelligent}
    D^{a}(|x|^{-k}f(x))=&\ |x|^{-k}D^{a}(f)(x)+\sum_{|b| <n }P_b(x)|x|^{-k-2(n-|b|)}D^b(f)(x) 
\end{align}
for homogeneous polynomials $P_b$ of degree $n-|b|$. This and \eqref{ineq:stupid} yield:
\[  | D^{a}(|x|^{-k}f(x))|\le C \sum_{0\le i\le n}\fnorm{f}_{i,k+n-i,r}\le C \fnorm{f}_{n,k,r},\]
for any $a\in \N_0^m$ with $|a|=n$. This proves the first part of the statement. 

For the other inequality, we first note that, by the Leibniz rule, we have:
\begin{align*}
    D^{a}(|x|^{-k-l}f(x))= &\ \sum_{a^1+a^2=a}D^{a^1}(|x|^{-l})D^{a^2}(|x|^{-k}f(x)).
    \end{align*}
The first term in the product can be estimated by:
\[ |D^{a^1}(|x|^{-l})|\le C |x|^{-l-|a^1|}.\]
By applying \eqref{eq: taylor} the term $D^{a^2}(|x|^{-k}f(x))$ can be bounded by:
\begin{align*}
    C \sum_{ |b|=l+|a^1|}|x|^{l+|a^1|}
     \int_{0}^1(1-s)^{l+|a^1|-1}|D^{a^2+b}(|\cdot|^{-k} f)(sx)|\dif s.
\end{align*}
Combining these estimates yields
\begin{align}\label{eq: strange taylor}
    |f|_{n,k+l,r}\le &\ C |f|_{n+l,k,r}
\end{align}

We prove the second inequality from the statement by induction on $n\in \N_0$. For $n=0$ and any $k\in \N_0$, we clearly have that
\[ \fnorm{f}_{0,k,r}=|f|_{0,k,r}.\]
Let us assume that the second inequality holds for all $k\in \N_0$ and all $n'<n$. 
Using \eqref{ineq:intelligent} and \eqref{eq: strange taylor}, we obtain that it holds also for $n$:
\begin{align*}
    \fnorm{f}_{n,k,r}&\leq |f|_{n,k,r} +C \sum_{0\le i<n}\fnorm{f}_{i,k+n-i,r}\\
    & \leq  |f|_{n,k,r} +C\sum_{0\le i<n}|f|_{i,k+n-i,r}\le\, C |f|_{n,k,r}.\qedhere
\end{align*}
\end{proof}
The lemma together with \eqref{eq: strange taylor} imply:
\begin{corollary}\label{corollary: degree shift}
For all $k,l,n\in \N_0$ and $f\in C^{\infty}_0(\overline{B}_r)$ we have
    \begin{align*}
    \fnorm{f}_{n,k+l,r}\le &\ C \fnorm{f}_{n+l,k,r}
\end{align*}
for a constant $0<C=C(k,l,n)$.
\end{corollary}

\subsection{Schwartz functions vs.\ flat functions}\label{section: Schwartz flat iso}
Recall that the space of Schwartz functions is defined as 
\[\mathcal{S}(\R^m):=\{f\in C^{\infty}(\R^m;\C)\, |\, \forall\,  k,n\in \N_0:\norm{f}_{n,k}<\infty\},\]
where we denote
\begin{align}\label{eq: Schwartz norm}
    \norm{f}_{n,k}:=\sup_{ x\in \R^m}\sup_{\substack{a\in \N_0^m:\ |a|\le n\\ l\in \N_0:\ l\le k}}
    \lvert x\rvert^{l}\lvert D^{a}f(x)\rvert \ \in \ [0,\infty].
\end{align}
With these norms, the space of Schwartz functions becomes a Fr\'echet space 
\[\big(\mathcal{S}(\R^m),\{\norm{\cdot}_{n,k}\}_{k,n\in \N_0}\big).\] 
The defining property of Schwartz functions is their behavior at infinity:
\begin{equation}\label{eq: behavior at infinity}
    \forall\, a\in \N^m_0,\, k\in \N_0 :\ 
\lim_{|x| \to \infty}|x|^{k}|D^{a}f(x)|=0.
\end{equation}
We refer to functions with this property as Schwartz-like at infinity, although they might not be defined on all of $\R^m$. 

We define norms on $C^{\infty}_0(\R^m;\C)$ similarly to \eqref{defintion:norms:appendix:flat}:
\[ \fnorm{f}_{n,k}:= \sup_{x\in \R^m\setminus \{0\}}\ \sup_{\substack{a\in \N_0^m:\ |a|\le n\\ l\in \N_0:\ l\le k}}|x|^{-l}|D^af(x)| \ \in \ [0,\infty]\]
for $k,n\in \N_0$. We denote the space of functions that are Schwartz-like at infinity and flat at the origin by $\mathcal{S}_0(\R^m)=\mathcal{S}(\R^m)\cap C^{\infty}_0(\R^m;\C)$. Then
\[(\mathcal{S}_0(\R^m),\{\norm{\cdot}_{n,k},\fnorm{\cdot}_{n,k}\}_{k,n\in \N_0})\] 
is a Fr\'echet space.

In order to compare the property of being Schwartz-like at infinity with that of being flat at the origin, we use the inversion $\rho$ on $\R^m \setminus\{0\}$ in the sphere of radius one, i.e., the map
\begin{equation}\label{eq: inversion}
    \begin{array}{cccc}
         \rho:& \R^m \setminus\{0\}&\to&\R^m \setminus\{0\} \\
         & x&\mapsto& \frac{x}{|x|^2}
    \end{array}
\end{equation}
Note that $\rho^2 =\id$. This map satisfies the following: 

\begin{proposition}\label{proposition: inversion}
The map $\rho$ induces a continuous linear involution
\[ \rho^* : \ \mathcal{S}_0(\R^m)  \to  \mathcal{S}_0(\R^m) \]
which satisfies the following estimates
\[ \fnorm{\rho^*(f)}_{n,k} \le C \norm{f}_{n,k+2n} \qquad \text{ and }\qquad \norm{\rho^*(f)}_{n,k+2n} \le C \fnorm{f}_{n,k}\]
where $0<C=C(k,n)$ for $k,n\in \N_0$.
\end{proposition}

\begin{proof}
In the proof we will use the inequalities:
\begin{equation}\label{eq: size of rho}
\forall\, x\neq 0 :\, |D^a\rho(x)|\leq C |x|^{-|a|-1},    
\end{equation}
for $a\in \N^m_0$ with $C=C(a)>0$.

Let $f\in \mathcal{S}_0(\R^m)$. The fact that $\rho^*(f)$ is well-defined, i.e., that it extends flatly over the origin and is Schwartz like at infinity, will follow once we prove the inequalities from the statement. 

The inequalities become equalities for $n=0$ and $k\in \N_0$:
\begin{align*}
    \fnorm{\rho^*(f)}_{0,k} =&\ \sup_{x\in \R^m\setminus \{0\}}\ \sup_{l\le k}|x|^{-l}|f(\rho(x))| \\
    =&\ \sup_{y\in \R^m\setminus \{0\}}\  \sup_{l\le k} |y|^{l}|f(y)|= \norm{f}_{0,k},
\end{align*} 
where we used $y=\rho(x)\in \R^m\setminus \{0\}$. 

For every $a \in \N_0^m$ with $|a|=n>0$, by using the general chain rule and \eqref{eq: size of rho}, we have for all $x\in \R^m\setminus \{0\}$:
\begin{equation}\label{eq: flat norm estimate}
 \begin{aligned}
    |D^{a}\rho^*(f)(x)|=&\ |\sum_{1\le|b|\le n}D^{b}(f)(\rho(x)){\sum}'\prod_{i=1}^{m}\prod_{j=1}^{b_i}D^{a_{ij}}(\rho)(x)|\\
    \le&\ C \sum_{1\le|b|\le n}\frac{|D^{b}(f)(\rho(x))|}{|x|^{n+|b|}}
\end{aligned}
\end{equation}
where $C =C(a)$ and the sum ${\sum}'$ runs over all $a_{ij}\in \N_0^m$ with
\[ |a_{ij}|\ne 0 \ \ \ \ \ \text{ and }\ \ \ \ \ \ \sum a_{ij}=a.\] 
Hence, by multiplying with $|x|^l$, for $l\leq k$, and setting $y=\rho(x)$, we obtain:
\begin{align*}
    \fnorm{D^{a}(\rho^*(f))}_{0,k} \le&\ C  \norm{f}_{n,k+2n}
\end{align*} 
which implies the first inequality. 

We prove now the second inequality. Let $a\in \N^m_0$ with $|a|=n>0$ and $k\in \N_0$. We use \eqref{eq: flat norm estimate} and the Taylor formula with integral remainder in \eqref{eq: taylor} for $d=n-|b|$ to obtain, for any $x\in \R^m\setminus \{0\}$ and $y=\rho(x)$, the following:
\begin{equation*}
 \begin{aligned}
    |x|^{2n+k}|D^{a}\rho^*(f)(x)| \le&\ C \sum_{1\le|b|\le n}|D^{b}(f)(\rho(x))||x|^{k+n-|b|}\\
   = &\ C \sum_{1\le|b|\le n}|D^{b}(f)(y)||y|^{|b|-k-n}\\
    \le &\ C (\sum_{1\le|b|< n}\sum_{|\tilde{b}|=n-|b|}|y|^{-k} \int_{0}^1(1-s)^{|\tilde{b}|-1}|D^{b+\tilde{b}} f(sy)|\dif s\\
    &\ \ \ \ \ +\fnorm{f}_{n,k})\\
    \le &\ C \fnorm{f}_{n,k}(1+\sum_{1\le|b|< n}\int_{0}^1s^k(1-s)^{n-|b|-1}\dif s)\\
    \le &\ C\fnorm{f}_{n,k}.
\end{aligned}
\end{equation*}
This implies the second inequality from the statement. 
\end{proof}

Let $C_r\subset \R^m$ denote the complement of the ball of radius $r>0$, i.e., 
\[C_r:=\R^m\setminus B_r\] 
and consider smooth functions on $C_r$ which are Schwartz-like at infinity: 
\[ \mathcal{S}(C_r ):=\{f\in C^{\infty}(C_r;\C)\, |\, \forall\,  k,n\in \N_0:\norm{f}_{n,k,r}<\infty\}.\]
Here the norms are defined for $k,n\in \N_0$ as in \eqref{eq: Schwartz norm}:
\[\norm{f}_{n,k,r}:=\sup_{ x\in C_r}\sup_{a\in \N^m_0: \, |a|\le n}\lvert x\rvert^k \lvert D^{a}f(x)\rvert.\]
For every $0<r$, we obtain the Fr\'echet space:
\[\big(\mathcal{S}(C_r ),\{\norm{\cdot}_{n,k,r}\}_{k,n\in \N_0}\big).\]

\begin{corollary}\label{corollary: inversion r}
The map $\rho $ induces a continuous linear isomorphism 
\[\rho_{B}^*:= (\rho |_{\overline{B}_r})^* : \mathcal{S}(C_{1/r})\to C^{\infty}_0(\overline{B}_r,\C)\]
with continuous inverse $\rho_C^*:= (\rho|_{C_{{1/r}}})^*$. Moreover, the map $\rho_B^* $ satisfies
\begin{align*}
    \fnorm{\rho_B^*(f)}_{n,k,r} \le &\ C 
    \norm{f}_{n,k+2n,1/r}
\end{align*}
for $k,n\in\N_0$ and $f\in \mathcal{S}(C_{1/r})$ and similarly $\rho_C^*$ satisfies the estimate
\begin{align*}
    \norm{\rho_C^*(g)}_{n,k+2n,1/r} \le&\ C
    \fnorm{g}_{n,k,r} 
\end{align*} 
for $k,n \in\N_0$, $g\in C^{\infty}_0(\overline{B}_r,\C)$ and a constant $C=C(k,n)$. 
\end{corollary}
\begin{proof}
The proof of the previous proposition applies by working with $\overline{B}_r$ and $C_{1/r}$ instead of the whole space.
\end{proof}

We will need the following version of Whitney's Extension Theorem: 

\begin{lemma}\label{lemma: extension}
There exists a linear map 
\[\epsilon : C^{\infty}_0(\overline{B}_1,\C)\hookrightarrow \mathcal{S}_0(\R^m)\]
such that, for each $f\in C_0^{\infty}(\overline{B}_1;\C)$, the function $\epsilon(f)$ is a compactly supported extension of $f$: 
\[ \epsilon(f)|_{\overline{B}_1}=f.\]
Moreover, the extension operator satisfies:
\begin{align*}
      \fnorm{\epsilon(f)}_{n,k}\le &\ C 
     \fnorm{f}_{n,k, 1} 
\end{align*} 
for all $n,k\in \N_0$ and constants $C=C(k,n)$.
\end{lemma}
\begin{proof}
Such extension results go back to Whitney's Extension Theorem \cite{WY}; for manifolds with boundary see \cite{Seeley}, and for the state of the art results see \cite{Baldi}. 

Our statement for the usual $C^n$-norms follows from the proof of the more general statement by Hamilton in \cite[Corollary II.1.3.7]{Ham}. By using the methods there, we can derive an explicit formula for the map $\epsilon$:
\[ \epsilon(f)(x):=\begin{cases}
f(x) &\text{ if } |x|\le 1\\
\chi(\frac{1}{|x|})\int_0^{\infty}\phi(t)(\chi f)(\frac{x}{|x|^{1+t}})\dif t&\text{ if } 1\le |x|
\end{cases}\]
for $f\in C^{\infty}_0(\overline{B}_1)$. Here $\chi\in C^{\infty}([0,\infty))$ denotes a function which satisfies
\[ \chi(u)=\begin{cases}
0&\text{ if } u\le \frac{1}{4}\\
1&\text{ if } \frac{1}{2}\le u
\end{cases}\]
and $\phi\in C^{\infty}([0,\infty))$ satisfies
\[ \int_0^{\infty}\phi(t)t^n\dif t=(-1)^n \qquad \forall \ n\in\N_0.\]
Since by \cite{Ham} the operator $\epsilon$ satisfies the estimates with respect to the $C^n$-norms, it can be easily shown that it satisfies also the estimates from the statement.
\end{proof}

%% file: Smoothing.tex
\subsection{Convolution, Fourier theory and regularization}\label{section: recall ft}

In this subsection we recall some basics of Fourier theory which will be used to construct the smoothing operators. 

An important operation on the Schwartz space is the \textbf{convolution} of two functions $f,g\in \mathcal{S}(\R^m)$, defined by:
\[(f*g)(x):=\int_{\R^m}f(x-y)g(y)\dif y \]
We have that $f*g\in \mathcal{S}(\R^m)$ and for any $a\in \N_0^m$ we have
\[ D^a(f*g)=(D^af)*g=f*(D^a g).\]
Recall also the \textbf{Fourier transform}
\[ \hat{f}(\xi):= \Big(\frac{1}{\sqrt{2\pi}}\Big)^{m} \int_{\R^m} f(x) e^{-i x\cdot \xi  }\dif x,\]
which is an automorphism of $\mathcal{S}(\R^m)$ with inverse:
\[ f(x)=\Big(\frac{1}{\sqrt{2\pi}}\Big)^{m} \int_{\R^m} \hat{f}(\xi) e^{i  x\cdot\xi }\dif \xi . \]
The Fourier transform satisfies the following properties:
\begin{equation}\label{eq: fundamental prop fourier}
    \begin{array}{rclcrcl}
        \hat{\hat{f}}(x)&=&f(-x) &\text{ and }&\widehat{f_{\lambda}}(\xi) &=& \hat{f}(\frac{\xi}{\lambda}), \\
        \hat{f}(\xi)&=&\bar{\hat{f}}(-\xi) \ \ \ & \Leftrightarrow& \ \ \ f&\in& \mathcal{S}(\R^m;\R),\\
        \widehat{D^a f}&=&(i\xi)^a\hat{f}& \text{ and }& \widehat{x^a f} &=&(-i)^{|a|}D^a\hat{f},\\
        \widehat{f*g} &=& (\sqrt{2\pi})^m \hat{f}\cdot \hat{g}&  \text{ and }& \widehat{f\cdot g}& =&  (\sqrt{2\pi})^m\hat{f}*\hat{g}
    \end{array}
\end{equation} 
for $f\in \mathcal{S}(\R^m)$, $\lambda \in \C^{\times}$, $a\in \N_0^m$ and $f_{\lambda}$ is defined by $f_{\lambda}(x):= \lambda^m f(\lambda \cdot x)$.


\subsection{Smoothing operators on the Schwartz space}\label{section: smoothing nash}

In this subsection we follow Nash \cite{Nash} to construct smoothing operators $\tilde{S}_t$ on the Schwartz space $\mathcal{S}(\R^m)$ for $0<t$. Nash proves estimates of $\tilde{S}_t$ for the standard norms $\norm{\cdot}_n$. Our goal is to provide estimates of $\tilde{S}_t$ with respect to the norms $\norm{\cdot}_{n,k}$ for all $k,n\in \N_0$. 

Fix a function $\chi \in C^{\infty}(\R)$ such that
\begin{align}\label{eq: bump}
    \chi(u)=\begin{cases}
		1&\text{ if }  u\le 1\\
		\text{monotone decreasing } &\text{ if } u\in (1,2)\\
		0&\text{ if }2\le u
	\end{cases}
\end{align} 
We define the function $K\in \mathcal{S}(\R^m)$ by its Fourier transform
\begin{align*}
    \hat{K}(\xi):=\Big(\frac{1}{\sqrt{2\pi}}\Big)^m \prod_{i=1}^m\chi(|\xi_i|).
\end{align*}

In the definition of $K$, we deviate slightly from Nash's construction, in an unessential way. The constant in the definition of $K$ ensures that:
\[ \int_{\R^m}K(x)=1.\]
For $t>0$, set $K_t(x):=t^mK(t\cdot x)$. Define the \textbf{smoothing operator} by
\begin{equation*}
	\begin{array}{cccc}
		\tilde{S}_t:&\mathcal{S}(\R^m)&\to &\mathcal{S}(\R^m)\\
		&f&\mapsto& K_t * f
	\end{array}
\end{equation*}
\begin{remark}
Note that $\tilde{S}_t$ preserves the $\R$-valued functions by \eqref{eq: fundamental prop fourier}. Therefore, so will all the other smoothing operators which we build below. 
\end{remark}

We show that these operators are indeed smoothings:

\begin{lemma}\label{eq: first smoothing original}
The operators $\tilde{S}_t$ satisfy:
\begin{align*}
	\norm{\tilde{S}_t(f)}_{n+l,k}\le C t^{l}\norm{f}_{n,k}
\end{align*}
for all $f\in \mathcal{S}(\R^m)$, all $k,l,n\in \N_0$ and all $1\leq t$, with $C=C(k,l,n)$.
\end{lemma}
\begin{proof}
For the proof note that for any $a\in \N_0^m$ we have
\[ D^aK_t(x)=t^{m+|a|}D^a(K)(tx).\]
Let $k, l, n\in \N_0$, and let $a\in \N^m_0$ with $|a|=n+l$. Decompose $a=a^1+a^2$, with $|a^1|=l$ and $|a^2|=n$. Using \eqref{eq: fundamental prop fourier} we obtain for any $f\in \mathcal{S}(\R^m)$:
\begin{align*}
	|x|^k|D^{a}(K_t&\, *f)(x)| = |x|^k||(D^{a^1}(K_t) *D^{a^2}(f))(x)|\\
	=& \ t^{l+m}|x|^{k}|\int_{\R^m}D^{a^1}(K)(t(x-y))D^{a^2}(f)(y)\dif y|\\
	\le & \ C t^{l+m}\Big(\int_{\R^m}|D^{a^1}(K)(t(x-y))||y|^{k}|D^{a^2}(f)(y)|\dif y\\
	& \ +t^{-k}\int_{\R^m}(t|x-y|)^{k}|D^{a^1}(K)(t(x-y))||D^{a^2}(f)(y)|\dif y\Big)
\end{align*}
where we used the estimate
\begin{align}\label{eq: triangle}
    |x|^k=|x-y+y|^k\le C(|x-y|^k + |y|^k)
\end{align}
for some constant $0<C(k)$. We estimate the first term by
\begin{align*}
    t^{l+m}\int_{\R^m}|D^{a^1}(K)&\,(t(x-y))||y|^{k}|D^{a^2}(f)(y)|\dif y\\
    \le &\ t^l \norm{f}_{n,k} \int_{\R^m}|D^{a^1}(K)(z)|\dif z\
    \le  C t^l \norm{f}_{n,k}
\end{align*}  
where we used the change of coordinates given by $z=t(x-y)$. The second term can similarly be estimated by using also that $t^{-k}\leq 1$:
\begin{align*}
    t^{l+m}\int_{\R^m}(t|x-y|)^{k}&\,|D^{a^1}(K)(t(x-y))||D^{a^2}(f)(y)|\dif y\\
    \le &\ t^{l} \norm{f}_{n,0} \int_{\R^m}|z|^k|D^{a^1}(K)(z)|\dif z \ \le C  t^{l} \norm{f}_{n,k}.
\end{align*}
This proves the statement.
\end{proof}

Next we show that the family $\tilde{S}_t$ approximates the identity as $t\to \infty$.

\begin{lemma}\label{eq: second smoothing original}
The operators $\tilde{S}_t$ satisfy:
\begin{align*}
	\norm{(\id - \tilde{S}_t)(f)}_{n,k}\le C t^{-l} \norm{f}_{n+l,k}
\end{align*}for all $f\in \mathcal{S}(\R^m)$, all $k,l,n\in \N_0$ and all $t\geq 1$, with $C=C(k,l,n)$.
\end{lemma}

\begin{proof}
First we claim that, for any $f\in \mathcal{S}(\R^m)$: 
\begin{align*}
\norm{f-\tilde{S}_t(f)}_{n,k}\xrightarrow{t\to \infty}0 \qquad \forall \ k,n\in\N_0.
\end{align*}
Note that, since $K$ is Schwartz function, for every $\eta>0$ we have
\begin{align*}
    \int_{|y|\geq \eta}|K_t(y)|\dif y
    =\int_{|x|\geq t \eta}|K(x)|\dif x
    \xrightarrow{t\to \infty} 0
\end{align*}
Let $\varepsilon>0$. For $R:= \norm{f}_{0,k+1}/\varepsilon$, we have that: \begin{align*}
    \sup_{x\in\R^m\setminus B_{R}} |x|^k|f(x)|\le \varepsilon
\end{align*}
For $x\in \R^m\setminus B_{2R}$, by using \eqref{eq: triangle}, we obtain
\begin{align*}
    |x|^k |(f-K_t*f)(x)|&\ = |x|^k|\int_{\R^m}K_t(y)(f(x-y)-f(x))\dif y|\\
    &\ \leq C\int_{B_{R}}(1+|y|^k|x-y|^{-k})|x-y|^k|K_t(y)f(x-y)|\dif y\\
    &\ \quad + |x|^k|f(x)|\int_{B_{R}}|K_t(y)|\dif y\\
    &\ \quad +  |x|^k|\int_{\R^m\backslash B_{R}}K_t(y)(f(x-y)-f(x))\dif y|\\
    &\ \le C\varepsilon +2\norm{f}_{0,0}|x|^k\int_{\R^m\setminus B_{R}}|K_t(y)|\dif y
    \xrightarrow{t\to \infty} \varepsilon C.
\end{align*}

Choose $\delta>0$ such that for all $x\in \overline{B}_{2R}$ and $y\in \overline{B}_{\delta}$:
\begin{align*}
(2R)^k|f(x)-f(x-y)|\le \varepsilon
\end{align*}
Then, for $x\in \overline{B}_{2R}$, we obtain:
\begin{align*}
    |x|^k |(f-K_t*f)(x)|&\le |x|^k|\int_{\R^m}K_t(y)(f(x-y)-f(x))\dif y|\\
    &\le \varepsilon +2\norm{f}_{0,0}|x|^k\int_{\R^m\setminus B_{\delta}}|K_t(y)|\dif y \xrightarrow{t\to \infty}\varepsilon
\end{align*}
This implies the claim made at the beginning.

In order to obtain the desired estimate we use the estimate
\begin{align*}
    \norm{\tilde{S}_{t_2}(f)-\tilde{S}_{t_1}(f)}_{n,k}\le \int_{t_1}^{t_2} \norm{\partial_t\tilde{S}_t(f)}_{n,k}\dif t
\end{align*}
for the limit $t_2\to \infty$. Taking the derivative with respect to $t$ of $\tilde{S}_t(f)$ gives
\[ \partial_t(K_t*f)=(\partial_tK_t)*f.\]
Note that the Fourier transform of $\partial_t K_t$ is given by
 \begin{equation*}
          \widehat{\partial_tK_t}(\xi)=\partial_t \hat{K_t}(\xi)=\partial_t \hat{K}(\xi/t)=
          -t^{-2}\sum_{j=1}^m|\xi_j|\chi^{(1)}(\frac{|\xi_j|}{t})\prod_{i\ne j}^m\chi(\frac{|\xi_i|}{t})
 \end{equation*}
We set $L:=\partial_tK_t|_{t=1}$. Similar as for $K_t$ we have 
\[ \partial_tK_t(x)=t^{m-1}L(t\cdot x)\]
We decompose $L$ in the following way: define the functions $\hat{L}_j$ by
\begin{equation*}
         \hat{L}_j(\xi):=|\xi_j|\chi^{(1)}(|\xi_j|)\prod_{i\ne j}^m\chi(|\xi_i|)\ \ \ \ \ \ \text{ for }\ j\in \{1,\dots,m\}. 
\end{equation*} 
Note that the $L_j$ satisfy the following properties
\begin{align*}
         L=\sum_{j=1}^mL_j, \qquad \hat{L}_j(\xi)=\bar{\hat{L}}_j(-\xi), 
\end{align*}
and that $\hat{L}_j$ is supported in $[-2,2]^m$ and vanishes for $|\xi_j|\leq 1$. Therefore we may define for $l\geq 0$ the functions $L_j^l$ via their Fourier transforms:
\[ \hat{L}_j^l := \big(\frac{1}{i\cdot \xi_j}\big) ^l \cdot \hat{L}_j\ \in \mathcal{S}(\R^m). \]
By \eqref{eq: fundamental prop fourier} the kernels $L_j^l$ satisfy for $k<l\in \N_0$ the following properties:
\begin{equation*}
    \begin{array}{c}
         \partial_j^kL_j^l=L_j^{l-k}\ \ \ \ \ \ \ \ \ \hat{L}_j^l(\xi)=\bar{\hat{L}}_j^l(-\xi) 
    \end{array} 
\end{equation*}
Hence we obtain for all $l, n\in \N_0$, all $f\in \mathcal{S}(\R^m)$ and any $a\in \N_0^m$ with $|a|=n$ the estimate
\begin{alignat*}{3}
	& |D^{a}(\partial_t K_t *f)(x)|&&\ =&&\ |(\partial_t K_t *D^a(f))(x)|\\
	& &&\ =&&\ t^{m-1}|\int_{\R^m}L(t(x- y))D^a(f)(y)\dif y|\\
	& && \ =&& \ t^{-1}|\int_{\R^m}\sum_{j=1}^mL_j(z)D^{a}(f)(x-\frac{z}{t})\dif z|\\
	& && \ =&& \ t^{-1}|\int_{\R^m}\sum_{j=1}^m\partial_{j}^l(L_j^l)(z)D^{a}(f)(x-\frac{z}{t}))\dif z|\\
	& && \ =&& \ t^{-1}|\int_{\R^m}\sum_{j=1}^mL_j^l(z)\partial_{j}^l(D^{a}(f)(x-\frac{z}{t}))\dif z|\\
	& && \ =&& \ t^{-l-1}|\int_{\R^m}\sum_{j=1}^mL_j^l(z)D^{a+l\cdot e_j}(f)(x-\frac{z}{t})\dif z|\\
	& && \leq && t^{m-l-1}\sum_{j=1}^m \int_{\R^m}|L_{j}^l(t(x-y))||D^{a+l\cdot e_j}(f)(y)|\dif y
\end{alignat*}
where we used the change of coordinates given by $z=t(x-y)$ and $e_j $ denotes the $j$-th standard vector in $\R^m$. Using \eqref{eq: triangle} as in the proof of the previous lemma we obtain, for $k\geq 0$:
\begin{align*}
	|x|^k|D^{a}(\partial_t K_t &*f)(x)|  \leq  t^{m-l-1}|x|^{k}\sum_{j=1}^m \int_{\R^m}|L_{j}^l(t(x-y))||D^{a+l\cdot e_j}(f)(y)|\dif y\\
 \leq\ &\ Ct^{m-l-1}\sum_{j=1}^m \int_{\R^m}|x-y|^{k}|L_{j}^l(t(x-y))||D^{a+l\cdot e_j}(f)(y)|\dif y\\
	 +\ &\  Ct^{m-l-1}\sum_{j=1}^m \int_{\R^m}|L_{j}^l(t(x-y))||y|^{k}|D^{a+l\cdot e_j}(f)(y)|\dif y\\
	 \leq \ &\  C\cdot t^{-l-1}(\norm{f}_{n+l,k}+t^{-k}\norm{f}_{n+l,0})
\end{align*}
which, using $t\geq 1$, leads to the estimate
\begin{align*}
	\norm{\partial_t K_t*f}_{n,k}\le C\cdot t^{-l-1}\norm{f}_{n+l,k}.
\end{align*}

Hence integration from $t$ to $\infty$ yields the statement.
\end{proof}

\subsection{Smoothing operators for flat functions}\label{section: smoothing flat}

In this section we use the smoothing operators on the Schwartz space from the previous section to construct smoothing operators $S_{t,r}$ on the space $C^{\infty}_0(\overline{B}_r)$. We define also a second family of operators $T_{s,r}$ on $ C^{\infty}_0(\overline{B}_r)$, which can be thought of as smoothing operators with respect to the $k$-index of the norms $\fnorm{\cdot}_{n,k,r}$. Finally, by combining these two, we build smoothing operators $S_{s,t,r}$ on $ C^{\infty}_0(\overline{B}_r)$ for $s,t>0$.

\subsubsection{Smoothing in the derivatives}
We first define $S_{t,1}:C^{\infty}_0(\overline{B}_1)\to C^{\infty}_0(\overline{B}_1)$ as the composition of the maps
\begin{equation*}
\begin{tikzpicture}[baseline= (a).base]
\node[scale=.85] (a) at (0,0){
\begin{tikzcd}
C^{\infty}_0(\overline{B}_1)\arrow{r}{\epsilon}& \mathcal{S}_0(\R^m)\arrow{r}{\rho^*}&  \mathcal{S}_0(\R^m)\arrow{r}{\tilde{S}_{t}}& \mathcal{S}(\R^m)\arrow{r}{\iota^*}& \mathcal{S}(C_1)\arrow{r}{\rho_{B}^*}& C^{\infty}_0(\overline{B}_1)
\end{tikzcd}
};
\end{tikzpicture}
\end{equation*}
where $\iota: C_1 \hookrightarrow \R^m$ is the inclusion and $\rho_B^*$ was defined in Corollary \ref{corollary: inversion r}. For $0<r $ denote by $m_r :\R^m\to\R^m $ the multiplication by $r$, i.e.
\[m_r(x) := \ r\cdot x\] 
Then we define the operators $S_{t,r}$ by
\begin{equation*}
\begin{array}{cccc}
S_{t,r}: & C^{\infty}_0(\overline{B}_r)&\to & C^{\infty}_0(\overline{B}_r)\\
&f &\mapsto& m_{\frac{1}{r}}^* \circ S_{t,1} \circ m_r^*(f)
\end{array}
\end{equation*}

\begin{proposition}\label{proposition: derivatiave smoothing}
The operators $S_{t,r}$ satisfy:
\begin{align*}
    \fnorm{S_{t,r}(f)}_{n+l,k,r} \le &\ C  t^{l}\fnorm{f}_{n,k+2 l,r}\\
    \fnorm{(S_{t,r}-\id)(f)}_{n,k+2l,r} \le &\  C t^{-l}\fnorm{f}_{n+l,k,r}
\end{align*}
for all $f\in C^{\infty}_0(\overline{B}_r)$, all $k,l,n\in \N_0$ and all $t\geq 1$, with $C=C(k,l,n,r)$ depending continuously on $r>0$.
\end{proposition}
\begin{proof}
By applying Corollary \ref{corollary: inversion r},  Lemma \ref{eq: first smoothing original}, Proposition \ref{proposition: inversion} and Lemma \ref{lemma: extension}, we obtain the following estimate for $S_{t,1}$:
\begin{align*}
    \fnorm{S_{t,1}(f)}_{n+l,k,1} = &\   \fnorm{\rho_B^* \circ \iota^* \circ \tilde{S}_{t}\circ \rho^*\circ  \epsilon (f)}_{n+l,k,1}\\
    \le &\  C \norm{\tilde{S}_{t}\circ \rho^*\circ \epsilon (f)}_{n+l,k+2(n+l)}\\
    \le &\  C  t^{l}\norm{\rho^*\circ \epsilon (f)}_{n,k+2(n+l)}\\
    \le&\ C  t^{l}\fnorm{\epsilon(f)}_{n,k+2l}\le C  t^{l}\fnorm{f}_{n,k+2l,1} 
\end{align*}
Similarly, by using Lemma \ref{eq: second smoothing original} instead of Lemma \ref{eq: first smoothing original}, we obtain that
\begin{align*}
    \fnorm{(S_{t,1}-\id)(f)}_{n,k+2l,1}= &\   \fnorm{\rho_B^* \circ \iota^* \circ (\tilde{S}_{t}-\id) \circ \rho^* \circ \epsilon (f)}_{n,k+2l,1}\\
    \le &\  C  \norm{(\tilde{S}_{t}-\id)\circ \rho^*\circ \epsilon (f)}_{n,k+2(n+l)}\\
    \le &\  C t^{-l}\norm{\rho^*\circ \epsilon (f)}_{n+l,k+2(n+l)}\\
    \le &\  C t^{-l}\fnorm{\epsilon(f)}_{n+l,k}\le C t^{-l}\fnorm{f}_{n+l,k,1}
\end{align*}
In order to obtain the bounds for $S_{t,r}$ note that for $f\in C^{\infty}_0(\overline{B}_R)$:
\[ \fnorm{f\circ m_r}_{n,k,R/r}\le \mathrm{max}(1,r^n) r^{k}\fnorm{f}_{n,k,R}\]
and we have a similar inequality for $m_{1/r}$. This shows that $C=C(k,l,n, r)$ can be taken to depend continuously on $0<r$.
\end{proof}

\subsubsection{Smoothing in the weights}
To define the operators $T_{s,r}:C^{\infty}_0(\overline{B}_r)\to C^{\infty}_0(\overline{B}_r)$ we proceed similarly as above. First we define $T_{s,1}$, for $s>0$, by
\begin{equation*}
    \begin{array}{cccc}
        T_{s,1}:&C^{\infty}_0(\overline{B}_1)&\to & C^{\infty}_0(\overline{B}_1)  \\
         & f&\mapsto &\chi( s |x|)f(x) 
    \end{array}
\end{equation*}
where $\chi\in C^{\infty}(\R)$ is such that
\begin{align*}
    \chi(u)=\begin{cases}
    0&\text{ if } u\le 1\\
    1&\text{ if } 2\le u
    \end{cases}
\end{align*}
The operators $T_{s,r}$ are then given by
\begin{equation*}
\begin{array}{cccc}
T_{s,r}: & C^{\infty}_0(\overline{B}_r)&\to & C^{\infty}_0(\overline{B}_r)\\
&f &\mapsto& m_{\frac{1}{r}}^*\circ T_{s,1} \circ m_r^*(f)
\end{array}
\end{equation*}
For these operators, we obtain the following estimates.
\begin{proposition}\label{proposition: weight smoothing}
The operators $T_{s,r}$ satisfy:
\begin{align*}
    \fnorm{T_{s,r}(f)}_{n,k+j,r} \le &\ C  s^{j}\fnorm{f}_{n,k,r}\\
    \fnorm{(T_{s,r}-\id)(f)}_{n,k,r} \le &\  C  s^{-j}\fnorm{f}_{n,k+j,r}
\end{align*}
for all $f\in C^{\infty}_0(\overline{B}_r)$, all $j,k,n\in \N_0$ and all $s>0$ with $C=C(j,k,n,r)$ depending continuously on $r>0$.
\end{proposition}
\begin{proof}
We proof the statement for $T_{s,1}$, from which the general case $T_{s,r}$ follows as in the proof of the previous proposition. 

For $a\in \N_0^m$ with $|a|=n>0$, and $l\in \mathbb{Z}$ we have:
\begin{align}\label{eq: boundedness}
    (s|x|)^{l}|D^{a}(\chi(s|x|))| \le &\ C|x|^{-n}\sum_{1\le |b|\le n}(s|x|)^{|b|+l}|(D^{b}\chi)(s|x|)|\leq C |x|^{-n}, 
\end{align}
where we use that, for all $i\in \Z$, $b \in \N_0^m$ with $|b|>0$: 
\[ \sup_{y\in \R }|y|^i|D^{b}\chi(y)| <\infty.\]
Note that \eqref{eq: boundedness} holds also for $a=0$ if $l\leq 0$.


Let $j,k,n\in \N_0 $ and $a\in \N_0^m$ with $|a|=n$. By using the above inequality and 
the Taylor formula with integral remainder from \eqref{eq: taylor}, we obtain:
\begin{align*}
    |x|&\,^{-k-j}|D^aT_{s,1}(f)(x)|\le |x|^{-k-j} \sum_{a^1+a^2=a}|D^{a^1}(\chi(s|x|))||D^{a^2}f(x)|\\
    = &\ s^{j}|x|^{-k}\sum_{a^1+a^2=a}(s|x|)^{-j}|D^{a^1}(\chi(s|x|))||D^{a^2}f(x)|\\
    \le &\ C s^{j}|x|^{-k}\sum_{a^1+a^2=a}|x|^{-|a^1|}|D^{a^2}f(x)|\\    
    \le &\ Cs^{j} \Big(\fnorm{f}_{n,k,1}+\sum_{\substack{a^1+a^2=a\\0<|b|=|a^1|}}\int_0^1(u|x|)^{-k}u^{k}(1-u)^{|a^1|-1}|D^{a^2+b}(f)(ux)|\dif u \Big)\\
    \le &\ C s^{j}\fnorm{f}_{n,k,1},
\end{align*}
for all $f\in C^{\infty}_0(\overline{B}_1)$. This implies the first inequality. 

For the second inequality, we write 
\[(T_{s,1}f-f)(x)=\kappa(s|x|) f(x),\quad \textrm{where}\quad \kappa(u):=1- \chi(u).\] 
Note that \eqref{eq: boundedness} holds with $\chi$ replaced by $\kappa$, for all $a$ with $|a|=n>0$ and all $l\in \mathbb{Z}$, and for $a=0$ it holds for $l\geq 0$. Therefore we can apply the same steps as in the proof of the first part, but with $-j\leq 0$ instead of $j\geq 0$, and we obtain the second inequality from the statement. 
\end{proof}

\subsubsection{Combined smoothing operators}\label{section:combined:smoothing}

By composing the operators $S_{t,r}$ and $T_{s,r}$, we obtain the smoothing operators $S_{t,s,r}$ for $s,t,r>0$ via: 
\[S_{t,s,r}:=T_{s,r} \circ S_{t,r} :  C^{\infty}_0(\overline{B}_r)\to C^{\infty}_0(\overline{B}_r).\]
For these operators we obtain the estimates
\begin{corollary}\label{corollary: smoothing flat functions}
The operators $S_{t,s,r}$ satisfy
\begin{align*}
    \fnorm{S_{t,s,r}(f)}_{n+l,k+j,r} \le &\ C  s^jt^{l}\fnorm{f}_{n,k+2 l,r}\\
    \fnorm{(S_{t,s,r}-\id)(f)}_{n,k+2l,r} \le &\  C  (t^{-l}\fnorm{f}_{n+l,k,r} +s^{-j}\fnorm{f}_{n,k+2l+j,r})
\end{align*}
for all $f\in C^{\infty}_0(\overline{B}_r)$, all $j,k,l,n\in \N_0$, all $s>0$ and $t\geq 1$, with $C=C(j,k,n,l,r)$ depending continuously on $r>0$.
\end{corollary}
\begin{proof}
For the first inequality we may directly apply Proposition \ref{proposition: derivatiave smoothing} and Proposition \ref{proposition: weight smoothing}. To prove the second inequality we use that
\begin{align*}
    \fnorm{(S_{t,s,r}-\id)(f)}_{n,k+2l,r} \le &\ \fnorm{(T_{s,r}\circ (S_{t,r}-\id)(f)}_{n,k+2l,r}\\
    &\ +\fnorm{(T_{s,r}-\id)(f)}_{n,k+2l,r}.
\end{align*} 
Now applying again Proposition \ref{proposition: derivatiave smoothing} and \ref{proposition: weight smoothing} yields the result.
\end{proof}

\subsection{Interpolation inequalities}\label{section: interpolation}

First recall the interpolation inequalities with respect to the standard $C^n$ norms $\norm{\cdot}_{n,r}$ for functions on $\overline{B}_r$ (see e.g., \cite{Ham}):
\begin{lemma}
We have that
\begin{align}\label{eq: standard interpolation}
\norm{f}_{n,r}\le &\  C \norm{f}^{\frac{l_2}{l_1+l_2}}_{n-l_1,r}\norm{f}^{\frac{l_1}{l_1+l_2}}_{n+l_2,r}
\end{align}
for all $f\in C^{\infty}_0(\overline{B}_r)$, all $l_1, l_2, n\in \N_0$, with $l_1\leq n$, with $C=C(n,l_1,l_2,r)$ depending continuously on $r>0$.
\end{lemma}
The equivalence of norms from Lemma \ref{lemma: equivalence} implies
\begin{corollary}
We have that
\begin{align}\label{eq: interpolation t}
    \fnorm{f}_{n,k,r}\le \ C \fnorm{f}^{\frac{l_2}{l_1+l_2}}_{n-l_1,k,r}\fnorm{f}^{\frac{l_1}{l_1+l_2}}_{n+l_2,k,r}
\end{align}
for all $f\in C^{\infty}_0(\overline{B}_r)$, all $k,l_1, l_2, n\in \N_0$, with $l_1\leq n$, with $C=C(n,k,l_1,l_2,r)$ depending continuously on $r>0$.
\end{corollary}
\begin{proof}
The result follows by applying the previous lemma to the function $|x|^{-k}f(x)$ and Lemma \ref{lemma: equivalence}.
\end{proof}

We also have the interpolation inequalities in the weight index, which we prove using the corresponding smoothing operators:
\begin{lemma}
We have that
\begin{align}\label{eq: interpolation s}
    \fnorm{f}_{n,k,r}\le &\ C \fnorm{f}^{\frac{j_2}{j_1+j_2}}_{n,k-j_1,r}\fnorm{f}^{\frac{j_1}{j_1+j_2}}_{n,k+j_2,r}
\end{align}
for all $f\in C^{\infty}_0(\overline{B}_r)$, all $k,j_1, j_2, n\in \N_0$, with $j_1\leq k$, where the constants $C=C(n,k,j_1,j_2,r)$ depend continuously on $r>0$.
\end{lemma}

\begin{proof}
The argument is standard (see e.g.\ \cite[Corollary III.1.4.2]{Ham}). By Proposition \ref{proposition: weight smoothing}:
\begin{align*}
    \fnorm{f}_{n,k,r}\le &\ \fnorm{T_{s,r}(f)}_{n,k,r} +\fnorm{(T_{s,r}-\id)(f)}_{n,k,r}\\
    \le &\   C ( s^{j_{1}}\fnorm{f}_{n,k-j_1,r} +s^{-j_2}\fnorm{f}_{n,k+j_2,r}). 
\end{align*}
The result follows by setting 
\[ s:= (\fnorm{f}_{n,k-j_1,r}^{-1}\fnorm{f}_{n,k+j_2,r})^{\frac{1}{j_1+j_2}}.\qedhere\]
\end{proof}